\documentclass{article}

 \usepackage[utf8]{inputenc}
 \usepackage{scalerel}
 \usepackage{enumerate}
 \usepackage{authblk}
 \usepackage{wrapfig}
\usepackage{tikz-cd}
\usepackage{amsmath,amscd}
\usepackage{mathtools}
 \usepackage{amsthm}
 \usepackage{amsfonts}
 \usepackage{color}
\usepackage[mathscr]{eucal}
 \usepackage{amssymb}
 \usepackage[margin=1in]{geometry}
\usepackage{comment}
\usepackage[]{lineno}

\makeatletter
\newtheorem*{rep@theorem}{\rep@title}
\newcommand{\newreptheorem}[2]{%
\newenvironment{rep#1}[1]{%
\def\rep@title{#2 \ref{##1}}%
\begin{rep@theorem}}%
{\end{rep@theorem}}}
\makeatother

 \newtheorem{theorem}{Theorem}[section]
 \newtheorem{proposition}[theorem]{Proposition}
 \newtheorem{corollary}[theorem]{Corollary}
 \newtheorem{lemma}[theorem]{Lemma}

 \newreptheorem{theorem}{Theorem}
 
 \theoremstyle{definition}
 \newtheorem{definition}[theorem]{Definition}
 
  \newtheorem{remark}[theorem]{Remark}

\usepackage{chngcntr}
\usepackage{apptools}
\AtAppendix{\counterwithin{theorem}{section}}

\definecolor{darkblue}{rgb}{0.0, 0.0, 0.8}
\definecolor{darkred}{rgb}{0.8, 0.0, 0.0}
\definecolor{darkgreen}{rgb}{0.0, 0.8, 0.0}
\definecolor{ncolor}{rgb}{0.8, 0.8, 0.0}

\newcommand{\class}{\mathcal{C}}
\newcommand{\fmm}[1]{\mathcal{#1}}
\newcommand{\mm}[1]{\mathscr{#1}}
\newcommand{\dgh}{d_{GH}}
\newcommand{\dghb}{d_{GH}^B}
\newcommand{\dha}{d_{H}}

\newcommand{\Z}{\mathbb{Z}}

\newcommand{\real}{\mathbb{R}}
\newcommand{\R}{\real}
\newcommand{\lip}{\mathcal{L}}

\newcommand{\dwa}[1]{d_{W,{#1}}}
\newcommand{\dwp}{d_{W,p}}
\newcommand{\dwi}{d_{W,\infty}}
\newcommand{\dgw}[1]{d_{GW,{#1}}}
\newcommand{\dgwp}{d_{{GW,p}}}
\newcommand{\dgwi}{\ensuremath{d_{GW,\infty}}}
\newcommand{\dto}{\ensuremath{\Rightarrow}}

\newcommand{\calX}{\mathcal{X}}
\newcommand{\calY}{\mathcal{Y}}
\newcommand{\calZ}{\mathcal{Z}}
\newcommand{\calE}{\mathcal{E}}
\newcommand{\calF}{\mathcal{F}}
\newcommand{\calR}{\mathcal{R}}

\newcommand{\calRB}{\mathcal{R}_B}
\newcommand{\calRBn}{\mathcal{R}^n_B}
\newcommand{\calDX}{\mathcal{D}_\mathcal{X}}
\newcommand{\calDY}{\mathcal{D}_\mathcal{Y}}
\newcommand{\calDXn}{\mathcal{D}_\mathcal{X}^n}
\newcommand{\calDYn}{\mathcal{D}_\mathcal{Y}^n}
\newcommand{\calFX}{\mathcal{F}_{\mathcal{X}}}
\newcommand{\calFY}{\mathcal{F}_{\mathcal{Y}}}
\newcommand{\calFXn}{\mathcal{F}_{\mathcal{X}}^n}
\newcommand{\calFYn}{\mathcal{F}_{\mathcal{Y}}^n}

\newcommand{\muX}{\mu_X}
\newcommand{\muY}{\mu_Y}

\newcommand{\psiX}{\psi_X}
\newcommand{\psiY}{\psi_Y}
\newcommand{\iX}{\iota_X}
\newcommand{\iY}{\iota_Y}
\newcommand{\piX}{\pi_X}
\newcommand{\piY}{\pi_Y}
\newcommand{\piZ}{\pi_Z}
\newcommand{\dX}{d_X}
\newcommand{\dY}{d_Y}
\newcommand{\dZ}{d_Z}
\newcommand{\dB}{d_B}
\newcommand{\mXY}{m_{X,Y}}
\newcommand{\cXY}{C(\muX,\muY)}
\newcommand{\dXY}{d_{X,Y}}

\newcommand{\dP}{dP}
\newcommand{\dPP}{d(P \otimes P)}
\newcommand{\dPn}{dP_n}
\newcommand{\dPPn}{d(P_n \otimes P_n)}
\newcommand{\sP}{\mathrm{supp}P}
\newcommand{\sPP}{\mathrm{supp}(P \otimes P)}

\newcommand{\costpPm}{\bigg( \int \mXY^p \dPP \bigg)^{1/p}}
\newcommand{\costpPd}{\bigg( \int \dXY^p \dP \bigg)^{1/p}}
\newcommand{\costiPm}{\sup_{\sPP}  \mXY}
\newcommand{\costiPd}{\sup_{\sP}\, \dXY}
\newcommand{\costiP}{\max \left\{ \frac{1}{2} \costiPm,
\, \costiPd  \right\}}
\newcommand{\costpP}{ \max \bigg\{ \frac{1}{2} \costpPm, \costpPd \bigg\}}

\newcommand{\unxpe}{U_{X,p}^{n,\epsilon}}
\newcommand{\unype}{U_{Y,p}^{n,\epsilon}}
\newcommand{\unxqe}{U_{X,q}^{n,\epsilon}}
\newcommand{\unyqe}{U_{Y,q}^{n,\epsilon}}
\newcommand{\cnxqe}{C_{X,q}^{n,\epsilon}}
\newcommand{\cnyqe}{C_{Y,q}^{n,\epsilon}}

\newcommand{\cat}[1]{\text{#1}}

\newcommand{\urysohnB}{\mathscr{U}_B}

\newcommand{\lawsB}{\mathcal{L}_B}

\begin{document}

\title{Universal Mappings and Analysis of Functional Data on Geometric Domains}
\author[1]{Soheil Anbouhi}
\author[2]{Washington Mio}
\author[3]{Osman Berat Okutan}
\affil[1]{Department of Mathematics and Statistics, Haverford College\\ \texttt{sanbouhi@haverford.edu}}
\affil[2]{Department of Mathematics, Florida Sate University\\ \texttt{wmio@fsu.edu}}
\affil[3]{Max Planck Institute for Mathematics in the Sciences\\ \texttt{okutan@mis.mpg.de}}
\date{ }

\maketitle
\begin{abstract}
This paper employs techniques from metric geometry and optimal transport theory to address questions related to the analysis of functional data on metric or metric-measure spaces, which we refer to as fields. Formally, fields are viewed as 1-Lipschitz mappings between Polish spaces, with the domain possibly equipped with a probability measure. We establish the existence and uniqueness, up to isometry, of Urysohn fields; that is, universal and homogeneous elements for this class. We prove a characterization theorem for Urysohn fields and show how they relate to a notion of Gromov-Hausdorff distance for fields. For metric-measure domains, we introduce a field analogue of the Gromov-Wasserstein distance and investigate its properties. Adapting the notion of distance matrices to fields, we formulate a discrete model and obtain an empirical estimation result that provides a theoretical basis for its use in functional data analysis. We also prove an analogue of Gromov's Reconstruction Theorem in this realm.
\end{abstract}

\bigskip


\medskip
{\em 2020 Mathematics Subject Clsssification:} 51F30, 60B05, 60B10.

\section{Introduction}\label{sec:intro}

This paper addresses problems in {\em functional metric geometry} that arise in the study of data such as signals recorded on geometric domains or the nodes of networks. Formally, these may be viewed as functions defined on metric spaces, sometimes equipped with additional structure such as a probability measure, in which case the domain is referred to as a {\em metric-measure space}, or simply $mm$-space. Datasets comprising such objects arise in many domains of scientific and practical interest. Metric spaces underlying functional data are typically finite, so we discuss some motivating examples of this nature. On a social network, a probability distribution $\mu$ on the set of nodes $V$ may be used to describe how influential the various members of the network are. The edges normally represent some form of direct interaction, but a metric $d$ on $V$, such as the shortest-path distance, the diffusion distance, or the commute-time distance \cite{coifman, martinez2019probing}, is useful in quantifying indirect interactions as well. The triple $(V,d,\mu)$ defines an $mm$-space. Attributes such as individual preferences, traits or characteristics may be viewed as a function $f \colon V \to B$, where $B$ is a metric space such as $\real^n$ for vector-valued attributes, or $\mathbb{Z}_2^n$ (binary words of length $n$) equipped with the $\ell_1$-norm for discrete attributes such as a like-or-dislike compilation of preferences. The quadruple $(V,d,\mu,f)$ is a functional $mm$-space that can be employed for data representation in many different scenarios. Social networks are dynamic, with individuals joining and leaving the network, their relevance changing over time, as well as their attributes \cite{tantipathananandh2007framework, sekara2016fundamental,doreian2013evolution}. This leads to a family of functional $mm$-space $(V_t, d_t, \mu_t, f_t)$ parametrized by time. To analyze, visualize and summarize structural and functional changes over time, it is important to define metrics that are sensitive to such changes and amenable to computation. One of the objectives of this paper is to develop representations and metrics for functional $mm$-spaces that brings us closer to this goal.

The study of the organization and shape of point clouds presented as finite metric spaces $(X, d_X)$ is a theme of central interest in geometric data analysis. The functional $mm$-spaces investigated in this paper can provide a useful tool in summarizing the organization of sub-communities of a point cloud and quantifying structural variation across point clouds. If the distance $d_X (x,y)$ is interpreted as a measure of interaction between $x,y \in X$, then tightly structured sub-communities of $X$ at a scale $r\geq 0$ can be modeled as non-empty subsets $\sigma \subseteq X$ such that $d_X (x,y) \leq r$, $\forall x,y \in \sigma$, which are the simplices of the Vietoris-Rips complex $V\!R (X,r)$. If $\tau$ is a face of $\sigma$, then the sub-community represented by $\tau$ is fully embedded in the sub-community represented by $\sigma$, so one may wish to only consider the sub-collection $K(X,r) \subseteq V\!R (X,r)$ formed by the maximal simplices of $V\!R (X,r)$; that is, those that are not faces of any other simplex. We refer to $K(X,r)$ as the {\em community hypergraph} of $X$ at scale $r\geq 0$. At scale $r=0$, maximal communities are the singletons $\{x\}$, $x \in X$, but they can be simplices of higher dimension if $d_X$ is a pseudo-metric. If $r_0 = \text{diam}\, X$, then $K(X,r_0)$ contains a single simplex $\sigma = X$ because $X$ itself represents a community at that scale. One can metrize $V\!R (X,r)$, for example, by defining the distance between two simplices $\sigma$ and $\tau$ to be the Hausdorff distance $d_H (\sigma, \tau)$ in $X$. Thus, $K(X,r)$ may be viewed as a (finite) metric space with the metric inherited from $V\!R (X,r)$. Letting $|\sigma|$ be the number of members in the sub-community represented by the simplex $\sigma$, we define a probability measure $\mu$ on $K(X,r)$ by $\mu(\sigma) = |\sigma|/N$, where
\[
N = \sum_{\sigma \in K(X,r)} |\sigma|
\]
is a normalizing constant. This distribution reflects how representative of $X$ each sub-community is. In addition, one may use centrality functions to quantify more explicitly the importance of each simplex $\sigma \in K(X,r)$. For $p\geq 1$, define the $p$-centrality function $\lambda_p \colon K(X,r) \to \real$ by
\[
\lambda_p (\sigma) := \left( \int_{K(X,r)} d_H^p (\sigma, \tau) d\mu(\tau)\right)^{1/p} = \left( \frac{1}{N} \sum_{\tau  \in K(X,r)} (d_H^p (\sigma,\tau)) |\tau| \right)^{1/p}
\]
(cf.\,\cite{hang2019topological}). The more central a sub-community $\sigma$ is, the lower the value of $\lambda_p (\sigma)$. Thus, we can approach the study of sub-communities of $X$ at scale $r$ via the community hypergraph viewed as a functional $mm$-space $(K(X,r), d_H, \mu, \lambda_p)$. For a family $\{(X_i, d_i) \colon 1 \leq i \leq n\}$ of point clouds, we can study variation in community organization via the associated family $\{(K(X_i,r), d_{H,i}, \mu_i, \lambda_{p,i}) \colon 1 \leq i \leq n\}$ of functional $mm$-spaces, potentially enhancing the analysis through the use of more sensitive metrics such as the ones developed in this paper.

Motivated by problems such as those outlined above, our primary goal is threefold: (i) to develop metrics that allow us to model and quantify variation in functional data, possibly with distinct domains; (ii) to investigate principled empirical estimations of these metrics; (iii) to construct a universal function that ``contains'' all functions whose domains and ranges are Polish (separable and complete metric) spaces, assuming Lipschitz regularity. The latter is much in the spirit of constructing universal spaces for structural data (metric spaces) whose investigation dates back to the early 20th century and are of classical interest in metric geometry \cite{urysohn1927espace, vershik2004random,katetov1986universal}.

Modeling structural data as geometric objects has been a subject of extensive study using techniques from areas such as metric geometry and optimal transport theory. Shape contrasts or similarities in a family of compact metric spaces can be quantified via metrics such as the Gromov-Hausdorff distance $d_{GH}$ \cite{burago2001course}. For $mm$-spaces, analogues of $d_{GH}$ such as the Gromov-Wasserstein distance \cite{memoli2011gromov, sturm2006geometry}, that highlight regions of larger probability mass, are also used for shape analysis. We study functional analogues of the Gromov-Hausdorff and Gromov-Wasserstein distances (cf.\,\cite{vayer2020fused}) and also investigate discrete representations of functional $mm$-spaces based on random finite subsets by means of distance-matrix representations that encode information about their shape.

Another important angle in the investigation of a shape class is the study of existence and uniqueness of a universal element. Given a class of metric spaces $\class$, a space $U$ is $\class$-universal if it contains an isometric image of each $X$ in $\class$. Fr\'echet showed that any Polish space can be isometrically embedded in the space $l^{\infty}$ of all bounded real sequences (with the sup norm) \cite{frechet1910dimensions}. Banach and Mazur showed that the same holds for the space $C[0,1]$ of real-valued continuous functions on the interval $[0,1]$ (with the sup norm) \cite{banach1987theory}. However, these spaces are external to the class of Polish spaces, as they are not separable themselves. A {\em Urysohn universal space} (or simply, Urysohn space) is a Polish space $U$ that is universal for Polish spaces and satisfies an additional homogeneity condition: any isometry between finite subsets of $U$ extends to an isometry of $U$ onto itself. Unlike more general universal spaces, Urysohn spaces are known to be unique up to isometry. For a more thorough discussion of a construction and properties of Urysohn spaces, we refer the reader to \cite{heinonen2003geometric} (see also \cite{katetov1986universal, vershik2004random}). A $\class$-universal space provides a framework for the study of the class $\class$. For example, the collection of compact subsets of the Urysohn space $U$ with the Hausdorff distance, modulo the action of the isometry group of $U$, is isometric to the space of isometry classes of compact metric spaces equipped with the Gromov-Hausdorff distance \cite{gromov2007metric}. This paper addresses functional counterparts of these problems for 1-Lipschitz mappings between Polish spaces, including the construction of Urysohn universal mappings that satisfy a homogeneity condition akin to the homogeneity of Urysohn spaces.

\subsection{Main Results}

We study the class of 1-Lipschitz maps $\phi \colon X \to B$, where $X$ and $B$ are Polish spaces. These mappings are the morphisms in the category $\cat{Met}$ of metric spaces restricted to Polish spaces. We refer to such 1-Lipschitz mappings as {\em $B$-valued fields} on $X$, or simply $B$-fields, and sometimes denote them as triples $\mm{X}=(X, B, \pi)$.

\medskip 
\noindent
{\bf Universal Urysohn Fields.} We first fix the target space $B$ and show that there exists a unique (up to isometry) {\em Urysohn universal $B$-field} $\pi \colon U \to B$ for the class of $B$-fields over Polish spaces. In Theorem \ref{thm:urysohncharacterization}, we prove that the following properties characterize a Urysohn $B$-field $\pi \colon U \to B$:
\begin{enumerate}[(i)]
\item for any $B$-field $\phi \colon X \to B$, there is an isometric embedding $\imath \colon X \to U$ such that $\pi \circ \imath = \phi$;
\item any (bijective) isometry $\phi \colon A \to A'$ between finite subsets of $U$ such that $\pi \circ \phi = \pi|_A$ can be extended to an automorphism of $\pi \colon U \to B$; that is, to a bijective isometry $\phi' \colon U \to U$ such that $\pi \circ \phi' = \pi$.
\end{enumerate}
Property (i) is referred to as universality, whereas property (ii) is a strong form of homogeneity. 

\smallskip

Subsequently, we introduce the notion of {\em ultimate Urysohn field} to address universality and homogeneity for the full class of 1-Lipschitz maps between Polish spaces; that is, allowing the target space $B$ to also vary. We prove an existence and uniqueness theorem for ultimate Urysohn fields, which are defined in terms of universality and homogeneity, but are characterized in Theorem \ref{thm:ultimateUrysohnCharacterization}, as follows. A $B$-field $\pi \colon U \to B$ is an ultimate Urysohn field if and only if it satisfies:
\begin{enumerate}[(i)]
\item $\pi$ is Urysohn universal for $B$;
\item $B$ is a Urysohn space. 
\end{enumerate}

As alluded to above, for compact metric spaces, a Urysohn space $U$ allows us to interpret the Gromov-Hausdorff distance as a Hausdorff distance in $U$ modulo the action of isometries. We prove an analogous result for compact fields; that is, fields with compact domains. To this end, in Section \ref{sec:mfields}, we introduce the notion of Gromov-Hausdorff distance between general compact fields.

Denote by $F (\mm{U})$ the space of compact subfields of a field $\mm{U}$ equipped with the Hausdorff distance in $\mm{U}$; that is, the Hausdorff distance between the domains of the compact subfields. Let $Aut (\mm{U})$ be the group of automorphisms of $\mm{U}$, and $d_F^\mm{U}$ the quotient metric induced by the Hausdorff distance on the quotient $F(\mm{U})/Aut(\mm{U})$.
\begin{reptheorem}{thm;fgh}
Let $\mm{U}$ be an ultimate Urysohn field. The moduli space $(\fmm{F}, \dgh)$ of isomorphism classes of compact fields is isometric to the quotient space $(F(\mm{U})/Aut(\mm{U}), d_F^\mm{U})$.
\end{reptheorem}

\medskip
\noindent
{\bf Metric-Measure Fields.} We study metric-measure fields ($mm$-fields), that is, 1-Lipschitz functions $f \colon X \to B$ between Polish spaces, where $X$ is also equipped with a Borel probability measure $\mu$. These are denoted as quadruples $\fmm{X}=(X, B,\pi, \mu)$. We develop a field analogue of the Gromov-Wasserstein distance that has been studied extensively for $mm$-spaces \cite{memoli2011gromov, sturm2006geometry}. For fields $\mathcal{X}$ and $\mathcal{Y}$, the Gromov-Wasserstein distance is denoted $\dgw{p}(\mathcal{X},\mathcal{Y})$, where $1 \leq p \leq \infty$ is a parameter, and is used to address convergence questions. Among other things, we show that there are equidistributed sequences $(x_n)$ and $(y_n)$ in $X$ and $Y$, respectively, such that
\[
\dgwi(\calX,\calY)= \max \left\{ \frac{1}{2} \sup_{i,j} \mXY(x_i,y_i,x_j,y_j),
    \, \sup_i \dXY(x_i,y_i) \right\}.
\]
This is further detailed in Theorem \ref{thm:gwuniform}. (A sequence $(x_i)$ in $X$ is $\mu$-equidistributed if the empirical measures $\sum_{i=1}^n \delta_{x_i}/n$ converge weakly to $\mu$.) With an eye toward principled empirical estimation of the distance between $mm$-fields, we introduce the notion of {\em extended distance matrices}, much in the way distance matrices are used to study $mm$-spaces (cf.\,\cite{gromov2007metric}). For an $mm$-field $\fmm{X}=(X, B,\pi, \mu)$ and a sequence $\{x_i\}$ in $X$, $i \geq 1$, form the countably infinite (pseudo) distance matrix $R = (r_{ij}) \in \real^{\mathbb{N} \times \mathbb{N}}$ and the infinite sequence $b = (b_i) \in B^{\mathbb{N}}$, where $r_{ij} = d_X (x_i, x_j)$ and $b_i = \pi(x_i) \in B$. We refer to the pair $(R, b)$ as the {\em augmented distance matrix} of $\fmm{X}$ associated with the sequence, which records the shape of the graph of $\pi$ restricted to the sequence. This construction let us define a Borel measurable mapping $F_{\fmm{X}} \colon X^\infty \to \real^{\mathbb{N} \times \mathbb{N}} \times B^\infty$, with both the domain and co-domain equipped with the weak topology. The pushforward of $\mu^\infty$ under $F_{\fmm{X}}$ yields a probability measure on $\real^{\mathbb{N} \times \mathbb{N}} \times B^\infty$ that we denote by
\begin{equation}
\fmm{D}_{\fmm{X}}:= (F_{\fmm{X}})_\ast (\mu^\infty) \,,
\end{equation}
and refer to as the {\em field curvature distribution} of $\fmm{X}$. This terminology is motivated by the notion of {\em curvature set} of a metric space \cite{gromov2007metric}. A similar construction for finite sequences $\{x_i\}$, $1 \leq i \leq n$, gives a measure $\fmm{D}^n_{\fmm{X}}$ on $\real^{n \times n} \times B^n$.

We prove a field version of Gromov's Reconstruction Theorem for metric-measure spaces that states that the field curvature distribution $\fmm{D}_{\fmm{X}}$ gives a faithful representation of $\fmm{X}$.

\begin{reptheorem}{thm:reconstruction}[Field Reconstruction Theorem]
Let $\fmm{X} = (X,d_X,\pi_X, \mu_X)$ and $\fmm{Y} = (Y,d_Y, \pi_Y, \mu_Y)$ be $mm$-fields over $B$ such that $\mu_X$ and $\mu_Y$ are fully supported. Then,
\[
\calX \simeq \calY \text{ if and only if } \calDX=\calDY.
\]
\end{reptheorem}
\noindent
Here, $\fmm{X}\simeq\fmm{Y}$ means that there exists a measure-preserving isometry between $\fmm{X}$ and $\fmm{Y}$.

Our main result supporting empirical estimation of the Gromov-Wasserstein distance between $mm$-fields is the following convergence theorem.
\begin{reptheorem}{thm: dWp(Dx,Dy)=dGW(infty)}
Let $\calX$ and $\calY$ be bounded $mm$-fields over $B$. Then, for any $1 \leq p \leq \infty$, we have
\begin{equation*}
    \lim_{n \to \infty} \dwp(\calDXn,\calDYn) = \dwp(\mathcal{D}_\mathcal{X},\mathcal{D}_\mathcal{Y}) = \dgwi(\mathcal{X},\mathcal{Y}) .
\end{equation*}
\end{reptheorem}


\subsection{Additional Remarks on Urysohn Space}
In \cite{mrowka1953solution}, Mro\'wka gave an example of an isometry between countable subsets of Urysohn space $U$ that cannot be extended to an isometry of $U$, whereas Huhunai{\v s}vili showed that isometries between compact subsets of $U$ can be extended to $U$. Additionally, in \cite{melleray2007geometry}, Melleray argued that compact homogeneity is the strongest general homogeneity property possible for Urysohn space. As it turns out, a metric space is Urysohn if and only if it has the following one-point-extension property for finite sets: given any finite subset $X \subseteq U$ and one-point metric extension $(X^\ast, d_X^\ast)$ of $X$, there is an isometric embedding of $X^\ast$ in $U$ that restricts to the inclusion map on $X$. 

A well-known construction of Urysohn space is due to Kat{\v e}tov \cite{katetov1986universal}. Starting with an arbitrary separable metric space $X_0$ (e.g., a one-point space), a sequence of separable metric extensions $\{X_i \colon i\geq 0\}$ is constructed inductively, as follows. For $i \geq 0$, $X_{i+1}$ is the space of all pseudo-metric one-point extensions of $X_i$. The space $X_i$ can be isometrically embedded into $X_{i+1}$ by mapping $x \in X_i$ to the one-point extension of $X_i$ that simply repeats the point $x$. The colimit $X_\infty := \varinjlim X_i$ is separable and its completion is a Urysohn space.

Vershik approached Urysohn space by means of (countably) infinite distance matrices, which may be thought of as metrics on the set $\mathbb{N}$ of natural numbers \cite{vershik2004random}. Distance matrices satisfying an {\em approximate one-point extension property} lead to Urysohn space by taking the completion of $\mathbb{N}$. Moreover, Vershik showed that such universal matrices are (topologically) generic and that a random matrix is universal in a probabilistic sense. Yet another approach uses somewhat different methodology from descriptive set theory and model theory to construct Urysohn space \cite{yaacov2015fraisse, kechris2005fraisse} and to address universality for spaces equipped with additional structures \cite{doucha2013universal}. 

\subsection{Organization}

Section \ref{sec:prelim} sets up some notation and terminology, recalls a few basic concepts and results from metric geometry and introduces their counterparts for fields defined on metric spaces. Section \ref{sec:urysohn} addresses existence, uniqueness and characterization of Urysohn fields. The notion of Gromov-Hausdorff distance between compact fields is introduced in Section \ref{sec:mfields} and we show how it relates to ultimate Urysohn fields. Metric-measure fields and various notions of distance between them are discussed in Section \ref{sec:mmfields}. Section \ref{sec:dmatrix} introduces a representation of $mm$-fields by distributions of infinite augmented distance matrices, proves a field reconstruction theorem based on these distributions and also addresses empirical approximation questions. Section \ref{sec:summary} closes the paper with a summary and some discussion. Some of the proofs are deferred to an appendix.


\section{Preliminaries} \label{sec:prelim}

In this section, we introduce some notation and review basic concepts and results from metric geometry 
(cf.\,\cite{burago2001course} and \cite{heinonen2003geometric}) that are used in the paper. We also introduce extensions of these basic concepts to the 
realm of functional data.

\begin{definition} An {\em extended metric} $d$ on a set $X$ is a metric that allows infinite distances. 
More precisely, a symmetric function $d: X \times X \to [0,\infty] $ which is zero on the diagonal, 
non-zero off the diagonal, and satisfies the triangle inequality.
\end{definition}

\begin{definition}[Kuratowski Space]
Given a set $X$, the {\em Kuratoswki space} $(\kappa(X),d_\infty)$ is the extended metric space whose underlying 
set is the collection of real valued functions on $X$ and the extended metric $d_\infty$ is defined by
\[
d_\infty(f,g):= \sup_{x \in X} |f(x)-g(x)| \,.
\] 
\end{definition}

If $(X, d_X)$ is a metric space, we define the following subspaces of $\kappa(X)$ that
interact with the metric structure of $X$:
\begin{enumerate}[(a)]
    \item $\lip^1(X):=\{f \in \kappa(X) \colon f \text{ is 1-Lipschitz} \}.$
    \item $\Delta(X):=\{f \in \kappa(X) \colon d_X(x,y) \leq f(x)+f(y), \forall x,y \in X \}.$
    \item $\Delta^1(X):=\lip^1(X) \cap \Delta(X)$.
\end{enumerate}

\begin{remark}
A function $f \in \Delta^1 (X)$ may be interpreted as a one-point (pseudo-metric) extension of $(X,d_X)$. Indeed, on the disjoint union $X^\ast = X \sqcup \{x^\ast\}$, define a (pseudo) metric $d_{X^\ast} \colon X^\ast \times X^\ast \to \real$ by $d_{X^\ast}|_{X\times X} = d_X$ and $d_{X^\ast} (x, x^\ast) =d_{X^\ast} (x^\ast,x):= f(x)$, $\forall x \in X$. Symmetry is clearly satisfied and conditions (a) and (b) above ensure that $d_{X^\ast}$ satisfies the triangle inequality.
\end{remark}

The following result is standard.

\begin{proposition}\label{prop:kuratowskiembeddig}
Let $(X,d_X)$ be a metric space. Then, $(\Delta^1(X),d_\infty)$ is a metric space and the 
map $X \to \Delta^1(X)$, defined by $x \mapsto d_X(x,\cdot)$, is an isometric embedding. This 
embedding is called the Kuratowski embedding. Furthermore, if $f \in \Delta^1(X)$, then $d_\infty(f,d_X(x,\cdot))=f(x)$. 
\end{proposition} 
\begin{proof}
See Proposition 2.2 of \cite{melleray2008some}.
\end{proof}

\begin{definition} Let $X$ be a set and $\mathcal{F}$ be a non-empty, pointwise bounded family of 
real-valued functions on $X$. The {\em pointwise supremum} of $\mathcal{F}$ is the function
$p\sup(\mathcal{F}): X \to \R$ given by
\[
p\sup(\mathcal{F})(x):= \sup_{f \in \mathcal{F}} f(x) \,.
\]
\end{definition}
\begin{proposition}\label{prop:psup}
Let $(X,d_X)$ be a metric space and $\mathcal{F}$ be a non-empty pointwise bounded 
family of real valued functions on $X$.
\begin{enumerate}[\rm (i)]
	\item If $\mathcal{F} \subseteq \lip^1(X)$, then $p\sup(\mathcal{F}) \in \lip^1(X)$.
	\item If $\mathcal{F} \subseteq \Delta(X)$, then $p\sup(\mathcal{F}) \in \Delta(X)$.
	\item If $\mathcal{F} \subseteq \Delta^1(X)$, then $p\sup(\mathcal{F}) \in \Delta^1(X)$.
\end{enumerate}
\end{proposition}
\begin{proof}
We use the abbreviation $F := p\sup(\mathcal{F})$.

\smallskip

\noindent
(i) For any $x, y \in X$, we have
\begin{equation}
F(x)-F(y) = \sup_{f \in \mathcal{F}} f(x)-F(y) \leq  \sup_{f \in \mathcal{F}} f(x) - f(y) \leq d(x,y) \,.
\end{equation}
Similarly, $F(y)-F(x) \leq d(x,y)$. Therefore, $F$ is 1-Lipschitz.

\medskip
\noindent
(ii) If $f \in \mathcal{F}$, then $F(x)+F(y) \geq f(x)+f(y) \geq d_X(x,y)$, for all $x,y \in X$.

\medskip
\noindent
(iii) Follows from (i) and (ii).
\end{proof}

\begin{definition}[Whitney-McShane Extension]
Let $(X,d_X)$ be a metric space and $\emptyset \ne A \subseteq X$. Given a $1$-Lipschitz function $f: A \to \R$, the \textit{Whitney-McShane} extension of $f$ is the function $\tilde{f}: X \to \R$ given by
\[
\tilde{f}(x):= \inf_{a \in A} (f(a)+d(a,x)) \,.
\]
\end{definition}

\begin{remark}\label{rem:multipleWhitneyMcshane}
It is simple to verify that the Whitney-McShane extension is the maximal $1$-Lipschitz 
extension of $f$. This implies that if $\emptyset \ne A \subseteq B \subseteq X$, taking two consecutive Whitney-McShane extensions, first from $A$ to $B$ and then from $B$ to $X$, is the same as taking the Whitney-McShane extension from $A$ to $X$.
\end{remark}

\begin{proposition}\label{prop:whitneyMcshane}
Let $(X,d_X)$ be a metric space and $A, B \subseteq X$ be non-empty subspaces. Denote by $\imath_A \colon \lip^1(A) \to \lip^1(X)$ and $\imath_B \colon \lip^1(B) \to \lip^1(X)$ the maps given by Whitney-McShane extension and by $r_A \colon \lip^1(X) \to \lip^1(A)$ and $r_B \colon \lip^1(X) \to \lip^1(B)$ the restriction maps. Then, the following statements hold:
\begin{enumerate}[\rm (i)]
    \item $\imath_A$ is an isometric embedding.
    \item $\imath_A(\Delta^1(A)) \subseteq \Delta^1(X)$.
    \item For any $f \in \lip^1(X)$, $d_\infty(\imath_A \circ r_A (f), \imath_B \circ r_B (f)) \leq 2 \, \dha ^X(A,B)$, where $d_H^X$ denotes the Hausdorff distance in $X$.
\end{enumerate}
\end{proposition}
\begin{proof}
(i) Let $f, g \colon A \to \real$ be 1-Lipschitz and $\tilde{f}, \tilde{g} \colon X \to \real$ denote their Whitney-McShane extensions, respectively. Clearly, $d_\infty(f,g) \leq d_\infty(\tilde{f},\tilde{g})$. For the reverse inequality, $\forall x \in X$, we have:
\begin{equation}
\begin{split}
    \tilde{f}(x)-\tilde{g}(x) & = \tilde{f}(x) - \inf_{a \in A} \,(g(a)+d_X(a,x))
   =\sup_{a \in A} \,(\tilde{f}(x)- g(a) - d_X(a,x)) \\
   & \leq \sup_{a \in A} \,(f(a)+ d_X(a,x)- g(a) - d_X(a,x)) \leq d_\infty(f,g) \,.
\end{split}
\end{equation}
Similarly, $\tilde{g}(x) - \tilde{f}(x) \leq d_\infty(f,g)$, $\forall x \in X$. Thus, $d_\infty(f,g) = d_\infty(\tilde{f},\tilde{g})$, as claimed.

\medskip

(ii) Let $f \in \Delta^1(A)$ and $\tilde{f} = \imath_A (f)$. For any $x, y \in X$, we have that
\begin{equation}
\begin{split}
    \tilde{f}(x)+\tilde{f}(y) &= \inf_{a,a' \in A} \,(f(a)+d_X(a,x)+f(a')+d_X(a',y)) \\
    &\geq \inf_{a,a' \in A} \,(d_X(a,a') + d_X(a,x) + d_X(a',y)) \geq d_X(x,y) \,.
\end{split} 
\end{equation}
Therefore, $\tilde{f} \in \Delta^1 (X)$.

\medskip

(iii) If $r > \dha^X(A,B)$, there exists a function $\phi: B \to A$ such that $d_X(b, \phi(b)) < r$, $\forall b$ in $B$. Then, for any $x \in X$, we have
\begin{equation}
\begin{split}
    \imath_A \circ r_A (f)(x) - \imath_B \circ r_B (f)(x) & = \imath_A \circ r_A (f)(x) - \inf_{b \in B} \,(f(b)+d_X(b,x)) \\
   & =\sup_{b \in B} \,(\imath_A \circ r_A (f)(x)- f(b) - d_X(b,x)) \\
   & \leq \sup_{b \in B} \,(f(\phi(b))+ d_X(\phi(b),x)- f(b) - d_X(b,x)) \\ 
   & \leq \sup_{b \in B} \,(f(\phi(b))-f(b)+ d_X(\phi(b),x) - d_X(b,x)) \leq 2r \,.
   \end{split}
\end{equation}
Similarly, $\imath_B \circ r_B (f)(x) - \imath_A \circ r_A (f)(x) \leq 2r$, $\forall x \in X$. Since $r > \dha(A,B)$ is arbitrary, the claim follows.
\end{proof}

\begin{definition}
Let $(X, d_X)$ be metric space and $\emptyset \ne A \subseteq X$. A $1$-Lipschitz function $f: X \to \R$ is said to be {\em supported in} $A$ if $f$ is the Whitney-MacShane extension of its restriction to $A$. If $f$ is supported in a finite subspace of $X$, $f$ is said to be {\em finitely supported}. We define 
\[
E(X):=\{f \in \Delta^1(X) \colon f\ \text{is finitely supported} \} .
\]
\end{definition}
\begin{remark}
Let $(X,d_X)$ be a metric space and $x \in X$. Then, $d_X(x,\cdot)$ is supported in $\{x\}$. Hence, the image of the Kuratowski embedding is contained in $E(X)$.
\end{remark}

The main reason we introduce the space $E(X)$ is that, in general, the separability of $X$ does not imply the separability of $\Delta^1(X)$ (e.g., if $X$ a countable discrete metric space), whereas $E(X)$ is separable (\cite{melleray2008some}, Section 2).

\begin{proposition}\label{prop:seperable}
If $(X,d_X)$ is a separable metric space, then $(E(X),d_\infty)$ is also separable.
\end{proposition}

\begin{proof} 
Let $C$ be a countable dense subset of $X$. By Remark \ref{rem:multipleWhitneyMcshane} and Proposition \ref{prop:whitneyMcshane}, we have that $\imath_C(E(C)) \subseteq E(X)$. First, we show that $\imath_C (E(C))$ is dense in $E(X)$.  Let $f \in E(X)$ be supported in the finite set $A= \{x_1,\ldots,x_n\}$ and $\epsilon >0$. There exists a subset $B=\{x_1',\ldots,x_n'\}$ of $C$ such that $d_X(x_i,x_i') < \epsilon/2$, for $1 \leq i \leq n$.  Let $g \in E(C)$ be the Whitney-McShane extension of $r_B(f)$ to $C$. By Remark \ref{rem:multipleWhitneyMcshane}, $\imath_C(g)=\imath_B \circ r_B(f) \in E(X)$.  Moreover, since $f$ is supported in $A$, we may  write $f = \imath_A \circ r_A (f)$. Thus, Proposition \ref{prop:whitneyMcshane} implies that
\begin{equation}
d_\infty(f, \imath_C(g))=d_\infty(\imath_A \circ r_A (f),\imath_B \circ r_B(f)) \leq \epsilon \,.
\end{equation}
This shows that $\imath_C(E(C))$ is dense in $E(X)$. It remains to show that $\imath_C(E(C))$ is separable. If $A$ is finite, $E(A) \subseteq \kappa(A)$ is separable, as the set of rational valued functions on $A$ is countable and dense in $\kappa(A)$. Since
\begin{equation}
\imath_C(E(C))= \bigcup_{A \subseteq C \text{ finite}} \imath_A(E(A)) \,,
\end{equation}
it follows that $\imath_C(E(C))$ is separable because it is a countable union of separable metric spaces.
\end{proof}

Spaces of type $E(X)$ can be used as the building blocks of a Urysohn space (\cite{heinonen2003geometric}, Theorem 3.4). As one of our goals is to extend the construction to Urysohn fields, we close this section by introducing functional analogues of  $\Delta^1(X)$ and $E(X)$. Recall that any $f \in \Delta_1 (X)$ may be interpreted as a one-point extension of $(X, d_X)$ to a pseudo-metric space $X^\ast = X \sqcup \{x^\ast\}$, where 
$d(x, x^\ast) = f(x)$, $\forall x \in X$. In the functional setting, we also have a 1-Lipschitz function  $\pi \colon X \to B$, where $B$ is separable. As such, a one-point extension of the functional data also requires an extension of $\pi$ to a 1-Lipschitz function on $X^\ast$. This motivates the next definition.

\begin{definition}
Let $\pi \colon X \to B$ be a $B$-field. We define the metric spaces $\Delta^1(\pi \colon X \to B)$ and $E(\pi \colon X \to B)$ by:
\begin{enumerate}[(i)]
    \item $\Delta^1(\pi \colon X \to B):=\{(f,b) \in \Delta^1(X) \times B:  d_B(b,\pi(x)) \leq f(x), \forall x \in X \}$ and
    \item $E(\pi \colon X \to B):=\{(f,b) \in E(X) \times B: d_B(b,\pi(x)) \leq f(x), \forall x \in X \}$,
\end{enumerate}
where the metric is given by $d (f_1, b_1), (f_2, b_2)) = \max \{ d_\infty (f_1, f_2), d_B (b_1, b_2)\}$.
\end{definition}

Note that the mapping $X \hookrightarrow E(\pi \colon X \to B)$ given by $x \mapsto (d_X(x,\cdot), \pi(x))$ is an isometric embedding. Moreover, if we let $\pi_B \colon E(\pi \colon X \to B) \to B$ denote the projection onto $B$, then the diagram
\begin{equation}
\begin{tikzcd}
X \arrow[rd, "\pi"'] \arrow[r, hook] & E(\pi \colon X \to B) \arrow[d, "\pi_B"] \\
& B
\end{tikzcd}
\end{equation}
commutes.

\begin{remark} \label{rem:relative} \
\begin{enumerate}[(i)]
\item If $X$ and $B$ are separable metric spaces, then so is $E(\pi\colon X \to B)$.
\item Let $\pi\colon X \to B$ be $1$-Lipschitz and $A \subseteq X$.  If $(f, b) \in \Delta^1(\pi\colon A \to B)$ and $\tilde{f}\colon X \to \R$ is the Whitney-McShane extension of $f$ to $X$, then $(\tilde{f}, b) \in \Delta^1(\pi\colon X \to B)$.
\end{enumerate}
\end{remark}


\section{Universal and Urysohn Fields} \label{sec:urysohn}

Recall that a Polish space $X$ is {\em universal} if any Polish space $Y$ admits an isometric embedding $\imath \colon Y \to X$. In general, however, universal spaces lack good homogeneity properties, as shown in Proposition \ref{prop:productNotUrysohn} below. Urysohn 
spaces are universal spaces that satisfy a strong form of homogeneity:  for any pair of isometric embeddings $\imath_1, \imath_2 \colon A \to X$, 
where $A$ is a finite metric space, there exists an isometry $\jmath \colon X \to X$ such that $\jmath \circ \imath_1 = \imath_2$. The reader may consult \cite{melleray2007geometry} for a detailed discussion of these properties. In this section, we investigate functional counterparts of universal and Urysohn spaces.

\subsection{Existence and Uniqueness} \label{sec:exst-unq}

\begin{definition}
A $B$-field $\mm{X} = (X,B,\pi)$ is {\em $B$-universal} if for any $B$-field $\mm{Y} = (Y,B, \phi)$, there is an isometric embedding $\imath \colon Y \to X$ such that $\pi \circ \imath = \phi$. In other words, the diagram
\[
\begin{tikzcd}
Y \arrow[rd, "\phi"'] \arrow[r, hook,  "\imath"] & X \arrow[d, "\pi"] \\
& B
\end{tikzcd}
\]
is commutative.
\end{definition}

\begin{proposition}\label{prop:universalproduct}
If $X$ is a universal metric space and $B$ is Polish, then the projection map $\pi_B: X \times B \to B$ is a $B$-universal field, where $X \times B$ is endowed with the metric 
\[
d ((x_1, b_1), (x_2, b_2)) = \max \, \{d_X (x_1,, x_2), d_B (b_1, b_2)\} .
\]
\end{proposition}

\begin{proof}
Let $\phi\colon Y \to B$ be a $B$-field. By the universality of  $X$, there exists an isometric embedding $\imath \colon X \to Y$. Since $\phi$ is $1$-Lipschitz, the map $\imath \times \phi\colon Y \to X \times B$ is an isometric embedding and the diagram
\begin{equation}
\begin{tikzcd}
Y \arrow[rd, "\phi"'] \arrow[r, hook,  "\imath \times \phi"] & X \times B \arrow[d, "\pi_B"] \\
& B 
\end{tikzcd}
\end{equation}
clearly commutes.
\end{proof}

\begin{remark}\label{rem:universalRetract}
A universal $B$-field $\pi \colon X \to B$ is a $1$-Lipschitz retraction. More precisely, by considering the identity map $B \to B$, we obtain an isometric embedding $\imath \colon B \to X$ such that the diagram
\begin{equation}
\begin{tikzcd}
B \arrow[rd, "id"'] \arrow[r, hook,  "\imath"] & X \arrow[d, "\pi"] \\
& B
\end{tikzcd}
\end{equation}
commutes. If we identify $B$ with $\imath (B)$, then $\pi$ is a retraction.
\end{remark}

\begin{definition}[Urysohn Field]
A $B$-field $\pi\colon U \to B$ is called a {\em Urysohn field over $B$} if for each finite  subspace $A$ of $U$ and $1$-Lipschitz map $\phi\colon A^\ast \to B$ defined on a one-point metric extension $A^\ast= A \sqcup \{a^\ast\}$, satisfying $\phi|_A = \pi|_A$, there exists  an isometric embedding  $\imath \colon A^* \to U$ such that the restriction $\imath|_A$ is the inclusion map and $\pi \circ \imath = \phi$.
\end{definition}

\begin{theorem}[Existence and Uniqueness of Urysohn Fields]\label{thm:urysohnmap}
If $B$ is a Polish space, then the following statements hold:
\begin{enumerate}[\rm (i)]
    \item there exists a Urysohn field $\pi\colon U \to B$;
    \item if $\pi\colon U \to B$ and $\pi'\colon U' \to B$ are Urysohn fields, then there exists a bijective isometry $\psi \colon U \to U'$ such that the diagram
\[
\begin{tikzcd}
U \arrow[rd, "\pi"'] \arrow[rr,  "\psi"] & & U' \arrow[ld, "\pi'"] \\
& B &
\end{tikzcd}
\]
is commutative.
\end{enumerate}
\end{theorem}
\begin{proof}
(i) {\em Existence.} Let $X_0=B$ and $\pi_0\colon X_0 \to B$ be the identity map. Inductively, define $X_{n+1} := E(X_n \to B)$ and let $\pi_{n+1}\colon X_{n+1} \to B$ be the projection $\pi_{n+1} (f,b) =b$. Then, the diagram
\begin{equation}
\begin{tikzcd}
X_0 \arrow[d, "\pi_0"] \arrow[r, hook] & X_1 \arrow[d, "\pi_1"] \arrow[r, hook] & 
X_2 \arrow[d, "\pi_2"] \arrow[r, hook] & \cdots \\
B  \arrow[r, "id"] &   B \arrow[r, "id"] &   B \arrow[r, "id"] & \cdots
\end{tikzcd}
\end{equation}
is commutative, where the mappings on the first row are isometric embeddings. By viewing the isometric embeddings $X_n \hookrightarrow X_{n+1}$ as inclusions, we may write the metric co-limit as $X_\infty = \cup_{n\geq 0} X_n$. We denote the limiting 1-Lipschitz map by $\pi_\infty: X_\infty \to B$. Let $U$ be the completion of $X_\infty$ and $\pi\colon U \to B$ be the $1$-Lipschitz extension of $\pi_\infty$. Note that $U$ is a Polish space by Remark \ref{rem:relative}(i). We now show that $\pi \colon U \to B$ is Urysohn. 

Let $A=\{u_1,\dots,u_n \}$ be a finite subspace of $U$, $A^\ast =\{u_1,\dots,u_n,u^\ast\}$ be a  one-point metric extension of $A$, and $\phi \colon A^\ast \to B $ a 1-Lipschitz extension of $\pi|_A$. Let $f \colon A \to \R$ to be the map $a \mapsto d_{A^\ast}(a,u^*)$ and $b :=\phi(u^*)$. Note that $(f,b) \in E(\pi \colon A \to B)$. Let $\tilde{f}$ be the Whitney-McShane extension of $f$ to $U$. By Remark \ref{rem:relative}(ii), $(\tilde{f},b) \in E(\pi \colon U \to B)$.

For each $\epsilon>0$, there exist $N^\epsilon \in \Z^+$ and $x^\epsilon_1,\dots, x^\epsilon_n \in X_{N^\epsilon}$ such that $d_U(u_i,x^\epsilon_i) \leq \epsilon$, for all $i=1,\ldots,n$.  Defining $f^\epsilon:=\tilde{f}|_{\{x^\epsilon_1,\ldots, x^\epsilon_n \}}$, we have that $(f^\epsilon,b) \in E(\pi \colon \{x^\epsilon_1,\ldots,x^\epsilon_n \} \to B)$. Let $\tilde{f}^\epsilon$ be the Whitney-McShane extension of $f^\epsilon$ to $X_{N^\epsilon}$. By Remark \ref{rem:relative}(ii), 
\begin{equation}
x^\epsilon:=(\tilde{f}^\epsilon,b) \in X_{N^\epsilon+1}=E(\pi \colon X_{N^\epsilon} \to B) \,.
\end{equation}
By Proposition \ref{prop:whitneyMcshane}, $d_U (x^\epsilon,x^{\epsilon'}) \leq 2(\epsilon+ \epsilon')$. Therefore, there exists $x \in U$ such that $\lim_{\epsilon \to 0} x^\epsilon = x$. Note that $\pi(x)=\lim_{\epsilon \to 0}\pi(x^\epsilon)=b$ and 
\begin{equation}
d_U(x,u_i) =\lim_{\epsilon \to 0}d_U(x^\epsilon,u_i)
          =\lim_{\epsilon \to 0}d_U(x^\epsilon,x^\epsilon_i)\\
          =\lim_{\epsilon \to 0} \tilde{f}(x^i_\epsilon)
          =f(u_i)=d_{A^\ast}(u_i,u^\ast) \,.
\end{equation}
Thus, by mapping $u^\ast$ to $x$, we get an isometric embedding $\imath \colon A^\ast \hookrightarrow U$ satisfying $\pi \circ \imath = \phi$, as desired.

\medskip

(ii) {\em Uniqueness.} Let $\pi \colon U \to B$ and $\pi' \colon U' \to B$ be Uryhson fields, and $C=\{x_1,x_2,\dots \}$ and $C'=\{x'_1, x'_2,\dots \}$ be countable dense subsets of $U$ and $U'$, respectively. We start by constructing nested sequences $\{C_n\}$ and $\{C'_n\}$ of finite subsets of $U$ and $U'$, respectively, and bijective isometries $\psi_n \colon C_n \to C'_n$ satisfying $\pi|_{C_n} = \pi'|_{C'_n} \circ \psi_n$. For $n=0$, define
\begin{equation}
C_0:=\emptyset, \ C'_0:= \emptyset, \ \text{and} \ \psi_0\colon C_0 \to C'_1 \,.
\end{equation}
Assuming that $C_n$, $C'_n$ and $\psi_n$ have been defined, we construct $C_{n+1}$, $C'_{n+1}$ and $\psi_{n+1}$, as follows. Since Uryhson fields have the one-point extension property, there exists $x' \in U'$ and  a bijective isometry $\psi'_n \colon C_n \cup \{x_{n+1}\} \to C'_n \cup \{x'\}$ that extends $\psi_n$ and for which the diagram
\begin{equation}
\begin{tikzcd}[row sep=0.35in]
C_n \cup \{x_{n+1}\} \arrow[rd, "\pi"'] \arrow[rr,  "\psi'_n"] &  & C'_n \cup \{x'\} \arrow[ld, "\pi'"] \\
& B &
\end{tikzcd}
\end{equation}
commutes. Similarly, there exists $x \in U$ and a bijective isometry $\psi''_n \colon C_n \cup \{x_{n+1},x\} \to C'_n \cup \{x',x'_{n+1}\}$ extending $\psi'_n$ such that the diagram
\begin{equation}
\begin{tikzcd}[row sep=0.35in]
C_n \cup \{x_{n+1},x\} \arrow[rd, "\pi"'] \arrow[rr,  "\psi''_n"] &  & C'_n \cup \{x',x'_{n+1}\} \arrow[ld, "\pi'"] \\
& B &
\end{tikzcd}
\end{equation}
is commutative. Let $C_{n+1}:=C_n \cup \{x_{n+1},x\}, C'_{n+1}:= C'_n \cup \{x',x'_{n+1}\}$ and $\psi_{n+1}:=\psi''_n$, which satisfies $\pi|_{C_{n+1}} = \pi'|_{C'_{n+1}} \circ \psi_{n+1}$ by construction. Then, in the limit, we obtain an isometry $\psi_\infty \colon C_\infty \to C'_\infty$ between the metric co-limits $C_\infty:= \cup_n C_n$ and $C'_\infty=\cup_n C'_n$ that extends $\psi_n$, for every $n$. Since $C_\infty$ and $C'_\infty$ are dense (as they contain $C$ and $C'$, respectively), the metric spaces $U$, $U'$ and $B$ are complete, and $\pi$ and $\pi'$ are $1$-Lipschitz, the map $\psi_\infty$ extends to  a bijective isometry $\psi \colon U \to U'$ such that $\pi=\pi' \circ \psi$.
\end{proof}

\begin{proposition}\label{prop:urysohnuniversal}
A Urysohn field over $B$ is a universal $B$-field.
\end{proposition}
\begin{proof}
Let $\pi\colon U \to B$ be a Urysohn field and $\phi\colon X \to B$ any $B$-field.  Choose an arbitrary countable dense subset $C=\{x_1,x_2,\ldots \}$ of $X$. Set  $X_0=\emptyset$ and $X_n=\{x_1,\ldots,x_n\}$, $\forall n>0$, and let $\imath_0$ be the (empty) embedding of $X_0$ into $U$. Inductively, assume that an embedding $\imath_n\colon X_n \to U$ such that $\phi|_{X_n}=\pi \circ \imath_n$ has bee constructed. Since $\pi$ is Urysohn, there exists an isometric embedding $\imath_{n+1}\colon X_{n+1} \to U$ extending $\imath_n$ such that $\phi|_{X_{n+1}}=\pi \circ \imath_{n+1}$. These embeddings combine to give an isometric embedding $\imath_\infty \colon C \to U$ such that the diagram
\begin{equation}
\begin{tikzcd}[column sep=large]
C \arrow[rd, "\phi"'] \arrow[r, hook, "\imath_\infty"] & U \arrow[d,"\pi"] \\
& B
\end{tikzcd}
\end{equation}
is commutative. Since $X$ is contained in the closure $\overline{C}$ of $C$, the desired embedding is obtained by extending $\imath_\infty$ to $\overline{C}$.
\end{proof}

Propositions \ref{prop:universalproduct} and \ref{prop:urysohnuniversal} imply that if $U$ is a Urysohn space, then the projection  $\pi_B \colon U \times B \to B$ is a universal field. In sharp contrast, $\pi_B$ is never a Uryhson field if $|B| > 1$.

\begin{proposition}\label{prop:productNotUrysohn}
Let $U$ and $B$ be Polish spaces and $U \times B$ be the product space endowed with the max metric. If $B$ has more than one element, then the projection map $\pi_B \colon U \times B \to B$ is not Urysohn.
\end{proposition}

\begin{proof}
Arguing by contradiction, assume that $\pi_B$ is Urysohn . Let $b_0, b_1$ be distinct points in $B$, $\delta:=d_B(b_0,b_1)$, and $u_0$ an arbitrary point in $U$. Since $U$ is a Urysohn space, there exists $u_1 \in U$ such that $d_U(u_0,u_1) < \delta$. Let $A:= \{(u_0,b_0), (u_1,b_1)\} \subseteq U \times B$, $A^\ast := A \sqcup \{a^\ast\}$ a one-point extension of $A$ satisfying
\begin{equation}
d_{A^\ast}((u_0,b_0),a^\ast)=2 \delta 
\quad \text{and} \quad 
d_{A^*}((u_1,b_1),a^\ast)=3 \delta \,,
\end{equation}
and $\phi \colon A^\ast \to B$ the $1$-Lipschitz extension of $\pi|_A$ given by $\phi(a^\ast)=b_0$. Since $\pi$ is a Uryhson field, there is $u \in U$ such that
\begin{equation}
d_{U \times B}((u,b_0),(u_0,b_0))= 2 \delta 
\quad \text{and} \quad
d_{U \times B}((u,b_0),(u_1,b_1))= 3 \delta \,.
\end{equation}
Since $U \times B$ is endowed with the max metric, this implies that $d_U(u,u_0)= 2 \delta$ and $d_U(u,u_1)= 3 \delta$, which is a contradiction since $d_U(u_0,u_1) < \delta$.
\end{proof}


\subsection{Characterization of Urysohn Fields}

\begin{definition}[Automorphisms and Matchings] \label{def:auto}
Let $\pi \colon X \to B$ be a $B$-field. 
\begin{enumerate}[(i)]
\item An {\em automorphism} of $\pi \colon X \to B$ is a bijective isometry $\psi \colon X \to X$ that satisfies $\pi \circ \psi = \pi$.
\item A {\em partial isometric matching} of $\pi \colon X \to B$ is a bijective isometry $\phi \colon A \to A'$ between  subspaces of $X$ that satisfies $\pi|_{A'} \circ \phi = \pi|_A$.
\item A partial isometric matching is said to be {\em finite} if $|A| = |A'|< \infty$.
\end{enumerate}
\end{definition}

\begin{proposition}\label{prop:homogeneous}
If $\pi\colon U \to B$ is a Uryhson field, then every finite partial isometric matching of $\pi$ can be extended to an automorphism of $\pi$.
\end{proposition}
\begin{proof}
Let $\phi \colon A \to A'$ be an isometric matching of $\pi \colon X \to B$, where $A=\{ x_1,\ldots,x_n\}$, $A'=\{x'_1,\ldots,x'_n \}$,  and $\phi(x_i)=x'_i$. Let $C=\{x_1,\ldots,x_n,x_{n+1},\ldots \}$ and $C'=\{x'_1,\ldots,x'_n,x'_{n+1},\ldots \}$ be countable  dense sets in $X$. The isometry extending $\phi$ is constructed as in the proof of the uniqueness part of Theorem \ref{thm:urysohnmap}.
\end{proof}

\begin{theorem}[Characterization of Uryhson fields]\label{thm:urysohncharacterization}
A $B$-field $\pi \colon U \to B$ is Urysohn if and only if it is universal and every finite isometric matching of $\pi$ can be extended to an automorphism of $\pi$.
\end{theorem}
\begin{proof}
The ``only if'' part follows from Proposition \ref{prop:urysohnuniversal} and Proposition \ref{prop:homogeneous}. To prove the converse statement let $A=\{u_1,\ldots,u_n \}$ be a finite subspace of $U$, $A^\ast=A \sqcup \{a^\ast\}$ a one-point extension of $A$, and $\phi \colon A^\ast \to B$ a 1-Lipschitz map extending $\pi|_A$. Since $\pi \colon U \to B$ is universal, there exists an isometric embedding $\imath \colon A^\ast \to U$ such that $\phi=\pi \circ \imath$.

Let $u'_i := \imath (u_i)$, $1 \leq i \leq n$, and $A'=\{u'_1,\ldots,u'_n \}$. Then, $\imath_A \colon A \to A'$ is a finite isometric matching of $\pi \colon X \to B$. By assumption, there is an automorphism $\psi \colon X \to X$ extending this matching. Let $u=\psi^{-1}(\imath (a^\ast))$. Note that $\pi(u)=\pi(\psi(u))=\pi(\imath (a^\ast))
=\phi(a^\ast)$ and $d_U(u,u_i)=d_U(\psi(u),\psi(u_i))=d_U(\imath(a^\ast),\imath(u_i))=d_{A^\ast}(u_i,a^\ast)$, as desired.
\end{proof}


\subsection{The Ultimate Urysohn Field }

\begin{definition}
A field $\pi \colon X \to B $ is called \textit{strongly universal} if $\pi$ is a $B$-universal field and $B$ is a universal space.
\end{definition}

\begin{definition} \label{def:uurysohn}
A field $\pi \colon U \to B$ is an {\em ultimate Urysohn field} if for any finite subspaces $A \subseteq U$ and $Z \subseteq B$ with $\pi(A) \subseteq Z$ and $1$-Lipschitz map $p \colon A^\ast \to Z^\ast$ between one-point extensions $A^\ast$  and $Z^\ast$ of $A$ and $Z$, respectively, there exist isometric embeddings $\imath \colon A^\ast \to U$ and $\jmath \colon Z^\ast \to B$ such that
\begin{enumerate}[(i)]
\item $\imath|_A$ and $\jmath|_Z$ are the inclusion maps;
\item $\pi \circ \imath = \jmath \circ p$\,; that is, the diagram
\[
\begin{tikzcd}[column sep=large]
A^\ast \arrow[d, "p"'] \arrow[r, hook,  "\imath"] & U \arrow[d, "\pi"] \\
Z^\ast \arrow[r, hook, "\jmath"'] & B
\end{tikzcd}
\]
is commutative.
\end{enumerate}
\end{definition}

\begin{definition} \label{def:matching}
Let $\pi \colon X \to B$ be a field. 
\begin{enumerate}[(i)]
\item A {\em double automorphism} of $\pi \colon X \to B$ is a pair of bijective isometries $\varphi \colon X \to X$ and $\psi \colon B \to B$ such that the diagram
\[
\begin{tikzcd}[column sep=large]
X  \arrow[r, "\varphi"] \arrow[d, "\pi"'] & X \arrow[d, "\pi"] \\
 B \arrow[r, "\psi"'] & B
\end{tikzcd}
\]
is commutative.

\item An {\em isometric double matching} of $\pi \colon X \to B$ is a pair of bijective isometries $\phi \colon A \to A'$ between subspaces of $X$ and $\theta \colon Z \to Z'$ between subspaces of $B$ such that $\pi(A) \subseteq Z$, $\pi(A') \subseteq Z'$ and the diagram
\[
\begin{tikzcd}[column sep=large]
A  \arrow[r, "\phi"] \arrow[d, "\pi|_A"'] & A' \arrow[d, "\pi|_{A'}"] \\
 Z \arrow[r, "\theta"'] & Z'
\end{tikzcd}
\]
commutes.

\item  An isometric double matching is called {\em finite} if the isometries $\phi$ and $\theta$ are between finite subspaces.
\end{enumerate}
\end{definition}

\begin{theorem}[Characterization of Ultimate Urysohn Fields]\label{thm:ultimateUrysohnCharacterization}
Let $\pi \colon U \to B$ be a field. The following statements are equivalent:
\begin{enumerate}[\rm (i)]
    \item $\pi$ is an ultimate Urysohn field.
    \item $\pi$ is a Urysohn $B$-field and $B$ is a Urysohn space.
    \item $\pi$ is a strongly universal field and every finite isometric double matching of $\pi$ can be extended to a double automorphism of $\pi$.
\end{enumerate}
\end{theorem}
\begin{proof}
$\text{(i)} \dto \text{(ii)}$. By definition, $\pi$ is a Uryhson field, so we need to show that $B$ is a Urysohn space. Let $Z$ be a finite subset of $B$ and  $Z^\ast = Z \sqcup \{z^\ast\}$ a one-point extension of $Z$. Set $A = \emptyset \subseteq U$ and let $A^\ast =\{a^\ast\}$ be a one-point extension of $A$. Define $p \colon A^\ast \to Z^\ast$ by $p (a^\ast) = z^\ast$, which is clearly 1-Lipschitz. Since $\pi$ is an ultimate Uryhson field, there is an isometric embedding $\jmath \colon Z^\ast \to B$ extending the inclusion $Z \hookrightarrow B$ satisfying the additional conditions of Definition \ref{def:uurysohn}. Thus, $B$ is a Urysohn space.

$\text{(ii)} \dto \text{(iii)}$. The strong universality of $\pi$ follows from Proposition \ref{prop:urysohnuniversal}. With the notation of Definition \ref{def:matching}, let a finite isometric double matching of $\pi\colon X \to B$ be given by a pair of isometries $\phi \colon A \to A'$ and $\theta \colon Z \to Z'$. Since $B$ is a Urysohn space, $\theta$ can be extended to a bijective isometry $\tilde{\theta} \colon B \to B$. Since $\pi$ is a Uryhson field and $\tilde{\theta} \colon B \to B$ is an  isometry, $\pi'\colon \tilde{\theta}\circ \pi \colon U \to B$ also is a Uryhson field. As in the proof of the uniqueness part of Theorem \ref{thm:urysohnmap}, we can extend the isometry $\phi \colon A \to A'$ (which satisfies $\pi'|_A=\pi \circ \phi$) to a bijective isometry $\tilde{\phi} \colon U \to U$ such that $\pi \circ \tilde{\phi} = \pi'$. The pair $(\tilde{\phi}, \tilde{\theta})$ gives the desired double automorphism of $\pi \colon U \to B$.

$\text{(iii)} \dto \text{(i)}$. Let $A \subseteq U$ and $Z \subseteq B$ be finite subspaces with $\pi(A) \subseteq Z$ and let $p \colon A^\ast  \to Z^\ast$ be a 1-Lipschitz map between one-point extensions of $A$ and $Z$ such that $p|_A = \pi|_A$. Since $B$ is a universal space, there exists an isometric embedding $\jmath \colon Z^\ast \to B$. The universality of the map $\pi \colon U \to B$ implies that there exists an isometric embedding $\imath \colon A^\ast \to U$ such that $\pi \circ \imath = \jmath \circ p$.

Let $A'= \imath(A)$ and $Z'= \jmath(Z)$. Then, the maps $\imath|_A \colon A \to A'$  and 
$\jmath|_A \colon Z \to Z'$ yield a finite isometric double matching of $\pi \colon U \to B$, which extends to a 
double automorphism $\tilde{\imath} \colon U \to U$ and $\tilde{\jmath} \colon B \to B$. The maps
$\tilde{\imath}^{-1} \circ \imath \colon A^\ast \to U$ and $\tilde{\jmath}^{-1} \circ \jmath \colon Z^\ast \to B$
give the desired isometric embeddings that extend the inclusions $A \subseteq U$ and $Z \subseteq B$, 
respectively. The commutativity of the diagram
\begin{equation}
\begin{tikzcd}[column sep=large]
A^\ast \ar[d, "p"'] \ar[r, hook, "\tilde{\imath}^{-1} \circ \, \imath"]& U \ar[d, "\pi"] \\
Z^* \ar[r, hook, "\tilde{\jmath}^{-1} \circ \, \jmath"'] & B
\end{tikzcd}
\end{equation}
follows from the construction.
\end{proof}


\section{Functional Gromov-Hausdorff Distance}\label{sec:mfields}

The primary goal of this section is to introduce a moduli space of compact fields equipped with a Gromov-Hausdorff distance and show how it relates to the Hausdorff distance between compact subfields of the Urysohn field, up to the action of automorphisms.

\begin{definition} 
Let $\mm{X}=(X,B,\pi_X)$ and $\mm{Y}=(Y,B,\pi_Y)$ be $B$-fields. 
\begin{enumerate}[(i)]
\item A mapping from $\mm{X}$ to $\mm{Y}$ over $B$, sometimes denoted $\Phi \colon \mm{X} \dto \mm{Y}$,  consists of a 1-Lipschitz mapping $\phi \colon X \to Y$ such that the diagram
\[
\begin{tikzcd}
X \arrow[rd, "\pi_X"'] \arrow[rr,  "\phi"] & & Y \arrow[ld, "\pi_Y"] \\
& B &
\end{tikzcd}
\]
commutes. We adopt the convention that uppercase letters denote maps between fields and the corresponding lowercase letters denote the associated mapping between their domains.
\item $\Phi \colon \mm{X} \dto \mm{Y}$ is an {\em isometric embedding} if $\phi: X\to Y$ is an isometric embedding of metric spaces.
\item $\Phi \colon \mm{X} \dto \mm{Y}$ is an {\em isometry} if $\phi: X \to Y$ is a bijective isometry. Moreover, we say that $\mm{X}$ and $\mm{Y}$ are {\em isometric}, denoted by $\mm{X} \simeq \mm{Y}$, if there is an isometry $\Phi$ between them. If $\mm{X} = \mm{Y}$, we refer to an isometry as an {\em automorphism} of $\mm{X}$.
\end{enumerate}
\end{definition}

\begin{definition} (Gromov-Hausdorff Distance for Fields) \label{def: functional GHB distance}
Let $\mm{X}=(X,B,\pi_X)$ and $\mm{Y}=(Y,B,\pi_Y)$ be compact $B$-fields. The {\em Gromov-Hausdorff distance} is defined by
\begin{equation*}
    \dghb(\mm{X}, \mm{Y}) := \inf_{Z,\Phi,\Psi} d^Z_H(\phi(X),\psi(Y)),
\end{equation*}
where the infimum is taken over all $B$-fields $\mm{Z}$ and isometric embeddings $\Phi \colon \mm{X} \dto \mm{Z}$ and $\Psi: \mm{Y} \dto \mm{Z}$ over $B$. 
\end{definition}
\begin{remark}
\
\begin{enumerate}[(i)]
\item The definition of $\dghb$ only involves the Hausdorff distance between the domains of the functions once embedded in $\mm{Z}$. This is due to the fact that fields are 1-Lipschitz functions, so that differences in function values are bounded by distances in the domain.
\item Similar to the case of compact metric spaces, the space $\fmm{F}_B$ of isometry classes of compact $B$-fields with the Gromov-Hausdorff distance is Polish (\cite{burago2001course}, Sec. 7.3). 
\end{enumerate}
\end{remark}
Henceforth, we denote the Urysohn field over $B$ by $\mm{U}_B = (U, B, \pi_{U})$. The existence and uniqueness of $\mm{U}_B$, up to isometry,  is guaranteed by Theorem \ref{thm:urysohnmap}. An important step in relating the Gromov-Hausdorff distance to the Hausdorff distance in $\mm{U}_B$ is to show that any isometry between compact subfields of $\mm{U}_B$ extends to an automorphism of $\mm{U}_B$. We begin with the following one-point extension lemma.

\begin{lemma}\label{compact injectivity}
Let $\mm{U}_B$ be a Urysohn field over $B$, $A \subseteq U$ a compact subset, and $f \colon A^\ast \to B$ a field defined on a one-point metric extension $A^\ast= A \sqcup \{a^\ast\}$ satisfying $f_A = \pi|_A$. Then, there exists an isometric embedding $\imath \colon A^\ast \to U$ over $B$ such that the restriction $\imath|_A$ is the  inclusion map and $\pi \circ \imath = f$. 
\end{lemma}
\begin{proof}
A proof can be found in Appendix \ref{sec:appendix}.
\end{proof}
\begin{proposition}\label{compact homogeneity}
If $\mm{U}_B$ is a Urysohn field and $A, B \subseteq U$ are compact subsets, then any partial isometric matching $\phi \colon A \to B$ (see Definition \ref{def:auto}) of $\mm{U}_B$ extends to an automorphism of $\mm{U}_B$.
\end{proposition}
\begin{proof}
Let $C=\{x_1, x_2, \ldots \}$ and $C'=\{x'_1, x'_2, \ldots\}$ be countable dense sets in $U\setminus A$ and $U \setminus B$, respectively. Using the one-point compact extension property established in Lemma \ref{compact injectivity} and a back-and-forth argument applied to $C$ and $C'$, as in the proof of the uniqueness part of Theorem \ref{thm:urysohnmap}, $\phi$ can be extended to an isometry of $\mm{U}_B$.
\end{proof}

\begin{lemma} \label{GH in U}
Let $\mm{X}=(X,B,\pi_X)$ and $\mm{Y}=(Y,B,\pi_Y)$ be compact $B$-fields and $\mm{U}_B=(U,B,\pi_U)$ be the Urysohn $B$-field. Then,
\begin{equation*}
    \dghb(\mm{X}, \mm{Y}) = \inf_{\Phi,\Psi} d^U_H(\phi(X),\psi (Y)),
\end{equation*}
where the infimum is taken over all isometric embeddings $\Phi\colon \mm {X} \dto \mm{U}_B$ and $\Psi \colon \mm{Y} \dto \mm{U}_B$ over $B$.
\end{lemma}
\begin{proof}
The inequality $ \dghb(\mm{X}, \mm{Y}) \leq \inf_{\Phi,\Psi} d^U_H(\phi(X),\psi (Y))$ follows from the definition of $\dghb$. To prove the reverse inequality, let $\epsilon >0$. We show that there are isometric embeddings  $\Phi \colon \mm{X} \dto \mm{U}_B$ and $\Psi \colon \mm{Y} \dto \mm{U}_B$ such that 
\begin{equation}
d^U_H(\phi(X), \psi (Y)) \leq \dghb(\mm{X}, \mm{Y}) +\epsilon \,.
\end{equation}
By definition of the Gromov-Hausdorff distance, there is a $B$-field $\mm{Z}$ and isometric embeddings $\Phi' \colon \mm{X} \dto \mm{Z}$ and $\Psi' \colon \mm{Y} \dto \mm{Z}$ over $B$ such that
\begin{equation}
d^Z_H(\phi'(X), \psi'(Y)) \leq \dghb(\mm{X},\mm{Y}) +\epsilon \,.
\end{equation}
By the universality of $\mathscr{U}_B$, there is an isometric embedding $\Lambda \colon \mm{Z} \dto \mm{U}_B$ over $B$. Letting $\Phi =\Lambda\circ\Phi'$ and $\Psi=\Lambda\circ\Psi'$, we have
\begin{equation} \label{eq:ghineq}
d^U_H(\phi(X), \psi(Y)) = d^Z_H(\phi'(X), \psi'(Y)) \leq \dghb (\mm{X},\mm{Y}) +\epsilon \,.
\end{equation}
Clearly, the same inequality holds if the left-hand side of \eqref{eq:ghineq} is replaced with the infimum over $\Phi$ and $\Psi$. Taking the limit as $\epsilon \to 0$, we obtain the desired inequality.
\end{proof}

Given an equivalence relation $\sim$ on a metric space $(X,d_X)$, the quotient metric is the maximal (pseudo) metric on $X/\!\!\sim$ that makes the quotient map $\pi \colon X \to X/\!\!\sim$ 1-Lipschitz. Let $F(\mm{U}_B)$ be the space of compact subfields of $\mm{U}_B = (U,B,\pi_U)$, equipped with the Hausdorff distance, and $Aut(B)$ the automorphism group of $\mm{U}_B$, which acts on $\mm{U}_B$ by isometries. On $F(\mm{U}_B)/Aut(B)$, by \cite[Lemma~3.3.6]{burago2001course}, the quotient metric may be expressed as
\begin{equation} \label{eq:autb}
d_F^B (\mm{X},\mm{Y}) =  \inf_{\Phi, \Psi \in Aut(B)} d_H^U (\phi (X),\psi (Y))) = \inf_{\Psi \in Aut(B)} d_H^U (X,\psi(Y)).
\end{equation}

\begin{theorem} \label{thm: F_B = F_I}
The moduli space $(\fmm{F}_B, \dghb)$ of isometry classes of compact $B$-fields, equipped with the Gromov-Hausdorff distance, is isometric to the quotient space $(F(\mm{U}_B)/Aut(B), d_F^B)$.
\end{theorem}
\begin{proof}
For any compact $B$-fields $\mm{X}$ and $\mm{Y}$, Lemma \ref{GH in U} shows that 
\begin{equation}
    \dghb(\mm{X},\mm{Y})= \inf_{\Phi, \Psi} d^{U}_H(\phi(X),\psi(Y)),
\end{equation}
where $\Phi \colon \mm{X} \dto \mm{U}_B$ and $\Psi \colon \mm{Y} \dto \mm{U}_B$ are isometric embeddings. By Proposition \ref{compact homogeneity}, any other isometric embeddings $\Phi' \colon \mm{X} \dto \mm{U}_B$ and $\Psi' \colon \mm{Y} \dto \mm{U}_B$ differ from $\Phi$ and $\Psi$ by the action of automorphisms of $\mm{U}_B$. This proves the claim.
\end{proof}

We conclude this section with an extension of Theorem \ref{thm: F_B = F_I} that allows the target space $B$ to also vary.

\begin{definition}
Let $\mm{X}=(X,B_X,\pi_X)$ and $\mm{Y}=(Y,B_Y,\pi_Y)$ be fields. An {\em isometric embedding} $\Phi \colon \mm{X} \dto \mm{Y}$ consists of a pair of isometric embeddings $f: X\to Y$ and $g: B_X\to B_Y$ of metric spaces such that the diagram
\begin{equation*}
\begin{tikzcd}[column sep=large]
X  \arrow[r, "f"] \arrow[d, "\pi_X"'] & Y \arrow[d, "\pi_Y"] \\
 B_X \arrow[r, "g"'] & B_Y
\end{tikzcd}
\end{equation*}
commutes. If both $f$ and $g$ are also surjective, $\Phi$ is called an {\em isometry}. $\mm{X}$ and $\mm{Y}$ are {\em isometric}, denoted by  $\mm{X} \simeq \mm{Y}$, if there exists an isometry between them.
\end{definition}
\begin{proposition} \label{prop:extu}
If $\mm{U} =(U, U, \pi)$ is an ultimate Urysohn field ($U$ is a Urysohn space), then $\mm{U}$ has the one-point extension property for compact subfields and any isometric double-matching between compact subfields extends to an automorphism of $\mm{U}$.
\end{proposition}
\begin{proof}
Identical to the proof of Proposition \ref{compact homogeneity}.
\end{proof}
\begin{definition} \label{def: functional GH distance}
Let $\mm{X}=(X,B_X,\pi_X)$ and $\mm{Y}=(Y,B_Y,\pi_Y)$ be compact fields. The {\em Gromov-Hausdorff distance} $\dgh(\mm{X}, \mm{Y})$ is defined as
\begin{equation*}
    \dgh(\mm{X}, \mm{Y}) := \inf_{Z,\Phi,\Psi} d^Z_H(\phi(X),\psi(Y)),
\end{equation*}
where the infimum is taken over all isometric embeddings $\Phi \colon \mm{X} \dto \mm{Z}$ and $\Psi \colon \mm{Y} \dto \mm{Z}$, where $\mm{Z}$ is a compact field. 
\end{definition}
Let $\fmm{F}$ be the moduli space of isometry classes of compact fields (with arbitrary Polish codomain $B$) equipped with the Gromov-Hausdorff distance. Arguing as in the case of metric spaces, one can show that this moduli space of fields is Polish. Let $\mm{U}$ be an ultimate Urysohn field, $F(\mm{U})$ be the space of compact subfields of $\mm{U}$, and $Aut (\mm{U})$ be the automorphism group of $\mm{U}$. We equip the quotient $F(\mm{U})/Aut (\mm{U})$ with a metric $d_F^\mm{U}$ defined as in \eqref{eq:autb}.

\begin{theorem} \label{thm;fgh}
Let $\mm{U}$ be an ultimate Urysohn field. The moduli space $(\fmm{F}, \dgh)$ of isomorphism classes of compact fields is isometric to the quotient space $(F(\mm{U)}/Aut(\mm{U}), d_F^\mm{U})$.
\end{theorem}
\begin{proof}
The proof is identical to that of Theorem \ref{thm: F_B = F_I} using the compact one-point extension property of Proposition \ref{prop:extu}.
\end{proof}

\section{Metric-Measure Fields}\label{sec:mmfields}

Recall that a metric measure space ($mm$-space) is a triple $(X, d, \mu)$, where $(X, d)$ is a Polish metric space and $\mu$ is a Borel probability measure on $X$. The next definition introduces an analogue for fields.

\begin{definition}
A {\em metric measure field}, or $mm$-field, is a quadruple $(X,B,\pi,\mu)$, where $X$ and $B$ are Polish metric spaces, $\pi \colon X \to B$ is a $1$-Lipschitz map, and $\mu$ is a Borel probability measure on $X$. Two $mm$-fields are {\em isomorphic} if there is a measure-preserving isometry between them.
\end{definition}

Given fields $(X,B,\pi_X)$ and $(Y,B,\pi_Y)$ over $B$, we let $\mXY \colon (X \times Y) \times (X \times Y) \to \R$ and $\dXY \colon X \times Y \to \R$ be the functions given by
\begin{equation}
\begin{split}
\mXY(x,y,x',y') &:=|\dX(x,x')-\dY(y,y')| \\
\dXY(x,y) &:= \dB(\piX(x),\piY(y)) \,.
\end{split}
\end{equation}
Note that
\begin{equation}
    \begin{split}
        |\mXY(x_1,y_1,x_1',y_1')-\mXY(x_2,y_2,x_2',y_2')| &\leq |\dX(x_1,x_1')-\dX(x_2,x_2')| + |\dY(y_1,y_1') - \dY(y_2,y_2')| \\
        &\leq \dX(x_1,x_2) + \dX(x_1',x_2') + \dY(y_1,y_2) + \dY(y_1',y_2') \\
        &\leq 4 \max\{\dX(x_1,x_2), \dY(y_1,y_2), \dX(x_1',x_2'), \dY(y_1',y_2') \}.
\end{split}
\end{equation}
Therefore, if we endow $(X \times Y) \times (X \times Y)$ with the product sup metric, then $\mXY$ is $4$-Lipschitz. Similarly, $\dXY$ is $2$-Lipschitz. Using this notation, we introduce a field version of the Gromov-Wasserstein distance in a manner similar to \cite{vayer2020fused}. 

\smallskip

Given $(X, d_X, \muX)$ and $(Y, d_Y, \muY)$, let $\cXY$ denote the set of all couplings between $\muX$ and $\muY$, the collection of all probability measures $\mu$ on $X \times Y$ that marginalize to $\mu_X$ and $\mu_Y$, respectively.
\begin{definition}[Gromov-Wasserstein Distance for Fields]\label{def:fwd}
Let $\calX=(X,B,\piX,\muX)$ and $\calY=(Y,B,\piY,\muY)$ be $mm$-fields. For $1 \leq p < \infty$, the {\em Gromov-Wasserstein distance} $\dgwp(\calX,\calY)$ is defined as
\begin{equation*}
\begin{split}
    \dgwp(\calX,\calY):= \inf_{P \in \cXY} \costpP .
\end{split}
\end{equation*}
For $p=\infty$,
\begin{equation*}
    \dgwi(\calX,\calY):= \inf_{P \in \cXY} \costiP.
\end{equation*}
\end{definition}

The fact that $\dgw{p}$ and $\dgw{\infty}$ are metrics can be argued as in the case of metric measure spaces (see (\cite{memoli2011gromov}, Theorem 5.1).

\begin{remark}\label{rem:support}
As $\mXY$ and $\dXY$ are continuous, its supremum over a set does not change after taking the closure of the set. Since the support is the smallest closest set of full measure, the  suprema in the definition of $\dgwi(\calX,\calY)$ are essential suprema. 
\end{remark}

\begin{proposition}\label{prop:gwrealization}
For each $1\leq p \leq \infty$, $\dgwp(\calX,\calY)$ is realized by a coupling. Furthermore,
\begin{equation*}
\lim_{p \to \infty} \dgwp(\calX,\calY)=\dgwi(\calX,\calY).
\end{equation*}
\end{proposition}
\begin{proof}
For each integer $n\geq 1$, there exists $P_n \in \cXY$ such that the expression on the right-hand side in the definition of $\dgwp(\calX,\calY)$ is $ \leq \dgwp(\calX,\calY) + 1/ n$. By Proposition \ref{prop:couplings} in the appendix, without loss of generality, we can assume that $P_n$ converges weakly to a coupling $P$, which also implies that $P_n \otimes P_n$ converges weakly to $P \otimes P$ (\cite[Theorem~2.8] {billingsley2013convergence}). We show that $P$ is an optimal coupling.

\smallskip
\noindent
{\bf Case 1.} Suppose that $1 \leq p < \infty$. Since $\mXY$ and $\dXY$ are continuous and bounded below, by \cite[Lemma~4.3]{villani2008optimal} we have
\begin{equation}
        \int \mXY^p \dPP \leq \liminf_n \mXY^p \dPPn 
        \quad \text{and} \quad
        \int \dXY^p \dP \leq \liminf_n \int \dXY^p \dPn \,.
\label{eq:pnp}
\end{equation} 
Using \eqref{eq:pnp} and the fact that for any sequences $(a_n)$ and $(b_n)$ of real numbers the inequality
\begin{equation}
\max \,\{\liminf_n a_n, \liminf b_n\} \leq \liminf_n \max \{a_n,b_n\} 
\end{equation}
holds, we obtain
\begin{equation}
\begin{split}
        \dgwp(\calX,\calY) &\leq \costpP \\
        & \leq \liminf_n \max \left\{ \frac{1}{2} \left( \int \mXY^p \dPPn \right)^{1/p}, \left( \int \dXY^p \dPn \right)^{1/p} \right\} \\
        &\leq \liminf_n \left( \dgwp(\calX,\calY) + 1/n \right) = \dgwp(\calX,\calY).
\end{split}
\end{equation}
This implies that $P$ realizes the Gromov-Wasserstein distance, as claimed.

\medskip
\noindent    
{\bf Case 2} For $p=\infty$, we adapt the proof of \cite[Proposition~3]{givens1984class} to the present setting. Note that if $1 \leq q \leq q' < \infty$, then
\begin{equation}
\dgw{q}(\calX,\calY) \leq \dgw{q'}( \calX,\calY) \leq \dgwi(\calX,\calY).
\end{equation}
Hence, we have
\begin{equation}
\lim_{q \to \infty} \dgw{q}(\calX,\calY) = \sup_q \dgw{q}(\calX,\calY) \leq \dgwi(\calX,\calY) .
\end{equation}

Let $(q_n)$ be a sequence of real numbers satisfying $q_n \geq 1$ and $q_n \to \infty$. Let $P_n$ be the optimal coupling realizing $\dgw{q_n}(\calX,\calY)$. Without loss of generality, we can assume that $P_n$ converges to $P$ weakly, which implies that $P_n \otimes P_n$ also converges weakly to $P \otimes P$. Let $0 \leq r < \costiP$ and set
\begin{equation}
U =\{(x,y,x',y') \colon \mXY (x,y,x',y') /2 > r \} \quad \text{and} \quad V = \{(x,y) \colon \dXY (x,y) >r \}.
\end{equation}
Either $P \otimes P (U)> 0$ or $P(V) > 0$, let us first assume that $P \otimes P (U) = 2m > 0$. By the Portmanteau Theorem \cite[Theorem~11.1.1]{dudley2018real},
\begin{equation}
2m \leq \liminf P_n \otimes P_n(U).
\end{equation}
By passing to a subsequence if necessary, we can assume that $P_n \otimes P_n(U) \geq m> 0$, for all $n$. We then have 
\begin{equation}
        \dgw{q_n}(\calX, \calY) \geq \frac{1}{2} \bigg(\int \mXY^{q_n} \dPPn \bigg)^{1/q_n}
        \geq r \, (P_n \otimes P_n (U))^{1/q_n}  \geq r \, m^{1/q_n}.
\end{equation}
Hence,
\begin{equation}
        \lim_{p \to \infty} \dgwp(\calX,\calY) = \lim_n \dgw{q_n}(\calX,\calY) \geq r.
\end{equation}
    Since $r < \costiP$ is arbitrary, we get
    \begin{equation}
    \begin{split}
        \costiP &\geq \dgwi(\calX,\calY) \\ &\geq \lim_{p \to \infty} \dgwp(\calX,\calY) \\  &\geq \costiP.
    \end{split}
    \end{equation}
This shows that the coupling $P$ realizes $\dgw{\infty}(\calX,\calY)$ and also proves the convergence claim. The case $P(V)>0$ is handled similarly.
\end{proof}

The next proposition establishes a standard relation between Gromov-Wasserstein and Wasserstein distances in the setting of $mm$-fields.

\begin{proposition} \label{prop:gwembedding}
Let $\calX=(X,B,\pi_X,\mu_X)$ and $\calY=(Y,B,\pi_Y, \mu_Y)$ be $mm$-fields over $B$. Suppose that $\piZ \colon Z \to B$ is 1-Lipschitz and $\iX \colon X \to Z$ and $\iY \colon Y \to Z$ are isometric embeddings satisfying $\piX=\piZ \circ \iota_X$ and $\piY= \piZ \circ \iY$. Then, for any $1 \leq p \leq \infty$, we have
\[
\dgwp(\calX,\calY) \leq \dwp((\iX)_*(\mu_X),(\iY)_*(\mu_Y))\,.
\]
\end{proposition}
\begin{proof}Let $1 \leq p < \infty$. Let $Q$ be the optimal coupling between $(\iX)_*(\mu_X)$ and $(\iY)_*(\mu_Y)$ realizing $\dwp((\iX)_*(\muX), (\iY)_*(\mu_Y))$. Since $\iX (X) \times \iY (Y)$ has full measure in $(Z \times Z, P)$, there is a measure $P$ on $X \times Y$ such that $(\iX \times \iY)_*(P)=Q$. Since $(\iX)_*$ and $(\iY)_*$ are injective, $P \in C(\muX,\muY)$. We have
\begin{equation}
\begin{split}
    \costpPm &= \bigg(\iint |\dZ(w,w')-\dZ(z,z')|^p dQ(w,z)dQ(w',z') \bigg)^{1/p}\\
    &\leq \bigg( \iint (\dZ(w,z)+\dZ(w',z'))^pdQ(w,z)dQ(w',z') \bigg)^{1/p} \\
    &\leq \bigg( \int \dZ(w,z)^pdQ(w,z) \bigg)^{1/p} + \bigg( \int \dZ(w',z')^pdQ(w',z') \bigg)^{1/p} \\
    &= 2 \dwp((\iX)_*(\muX), (\iY)_*(\muY)).
\end{split}
\end{equation}
Similarly,
\begin{equation}
\begin{split}
    \costpPd &= \bigg(\int \dB(\piZ(w),\piZ(z))^p dQ(w,z) \bigg)^{1/p} \\
    &\leq \bigg(\int \dZ(w,z)^p dQ(w,z) \bigg)^{1/p} = \dwp((\iX)_*(\muX), (\iY)_*(\muY)).
\end{split}
\end{equation}
Hence,
\begin{equation}
    \dgwp(\calX,\calY) \leq \costpP \leq \dwp((\iX)_*(\muX), (\iY)_*(\muY)).
\label{eq:lim}    
\end{equation}
Letting $p \to \infty$ in \eqref{eq:lim}, we get $\dgwi(\calX,\calY) \leq \dwi((\iX)_*(\muX), (\iY)_*(\muY))$.
\end{proof}

Recall that, in a metric measure space $(X,d_X, \mu_X)$, a sequence $(x_n)$ is called uniformly distributed (u.d.) if  $\sum_{i=1}^n \delta_{x_i}/n \to \mu_X$ weakly. Let $U_X$ denote the set of uniformly distributed sequences in $X$. It is known that $U_X$ is a Borel set in $X^\infty$ and $\mu_X^\infty(U_X)=1$ \cite{kondo2005probability}.

The next result provides a characterization of $\dgwi(\calX,\calY)$ in terms of uniformly distributed sequences.

\begin{theorem}\label{thm:gwuniform}
Let $\fmm{X} = (X, d_X, \pi_X, \mu_X)$ and $\fmm{Y} = (Y, d_Y, \pi_Y, \mu_Y)$ be bounded $mm$-fields over $B$. Then,
    $$\dgwi(\calX,\calY)= \inf_{\tiny \begin{matrix} (x_n) \in U_X \\ (y_n) \in U_Y \end{matrix}} \max \left\{ \frac{1}{2} \sup_{i,j} \mXY(x_i,y_i,x_j,y_j),
    \, \sup_i \dXY(x_i,y_i) \right\}.$$
Furthermore, there are sequences $(x_n) \in U_X$ and $(y_n) \in U_Y$ that realize the infimum on the right-hand side.
\end{theorem}
\begin{proof}
 Let us denote the infimum on the right-hand side by $\alpha$ and let $P$ be an optimal coupling realizing $\dgwi(\calX,\calY)$. Then,
 \begin{equation}
 \dgwi(\calX,\calY)=\costiP .
 \end{equation}
Let $(x_n,y_n)$ be an equidistributed sequence with respect to $P$ in $\sP$. Then $(x_n) \in U_X$ and $(y_n) \in U_X$, and we have 
    \begin{equation}
    \begin{split}
        \alpha &\leq  \max \left\{ \frac{1}{2} \sup_{i,j} \mXY(x_i,y_i,x_j,y_j), \, \sup_i \dXY(x_i,y_i) \right\} \\
        &\leq \costiP = \dgwi(\calX,\calY).
    \end{split}
    \end{equation}
Now we prove the inequality in the other direction. Let $p \geq 1$, $\epsilon>0$, $(x_n) \in U_X$ and $(y_n) \in U_Y$. Let $\calE_n$ be the $mm$-field with domain $E_n= \{1,\dots,n\}$ equipped with the (pseudo)-metric $d_{E_n}(i,j)=d_X(x_i,x_j)$, normalized counting measure, and the 1-Lipschitz function $\pi_E(i)=\pi_X(x_i)$. Similarly, define $\calF_n$ using $(y_n)$. By Proposition \ref{prop:gwembedding}, we have
\begin{equation}
\dgwp(\calX, \calE_n) \leq \dwp(\mu_X, \sum_{i=1}^n \delta_{x_i}/n) \quad \text{and} \quad \dgwp(\calY, \calF_n) \leq \dwp(\mu_Y, \sum_{i=1}^n \delta_{y_i}/n) .
\end{equation}
Since the Wasserstein distances in the above inequalities converge to $0$ as $n \to \infty$ (\cite[Theorem~6.9]{villani2008optimal}), we can choose $n$ large enough so that the Gromov-Wasserstein distances in the above inequalities are $< \epsilon$. If we use diagonal coupling $P \in C(\calE_n, \calF_n)$ given by $P({i,i})=1/n$, we get
\begin{equation}
\dgwp(\calE_n,\calF_n) \leq  \max \left\{ \frac{1}{2} \sup_{i,j} \mXY(x_i,y_i,x_j,y_j), \, \sup_i \dXY(x_i,y_i) \right\}.
\end{equation}
This, in turn, implies that 
 \begin{equation}
\dgwp(\calX, \calY) \leq \max \left\{ \frac{1}{2} \sup_{i,j} \mXY(x_i,y_i,x_j,y_j), \, \sup_i \dXY(x_i,y_i) \right\} + 2\epsilon.
\end{equation}
Since $(x_n) \in U_X$, $(y_n) \in U_Y$, and $\epsilon>0$ is arbitrary, we get $\dgwp(\calX,\calY) \leq \alpha$. As $p \geq 1$ is arbitrary, Proposition \ref{prop:gwrealization} implies that
\begin{equation}
\dgwi(\calX,\calY)=\lim_{p \to \infty} \dgwp(\calX,\calY) \leq \alpha .
\end{equation}
This also shows that the sequences $(x_n) \in U_X$ and $(y_n) \in U_Y$ constructed above realizes the infimum.
\end{proof}

\section{Augmented Distance Matrices} \label{sec:dmatrix}

Given a sequence $(x_n)$ is an $mm$-space $(X, d, \mu)$, one can form an associated (infinite) distance matrix $D = (d_{ij})$, where $d_{ij} = d(x_i, x_j)$. Gromov's Reconstruction Theorem \cite{gromov2007metric} states that the distribution of all distance matrices for $(X,d,\mu)$ with respect to the product measure $\mu^\infty$ is a complete invariant. This section introduces {\em augmented distance matrices} to establish a similar result for $mm$-fields and also studies relationships between the Gromov-Wasserstein distance between $mm$-fields and the Wasserstein distance between the corresponding augmented distance matrix distributions. For an integer $n>0$, let
\begin{equation}
\calR^n:=\{(r_{ij}) \in \R^{n \times n}: r_{ij}=r_{ji}, r_{ii}=0, r_{ik} \leq r_{ij}+r_{jk} \}
\end{equation}
denote the space of all $n \times n$ (pseudo) distance matrices equipped with the metric
\begin{equation}
\rho_n ((r_{ij},b_i),(r'_{ij},b'_i)):=\max \, (\frac{1}{2}\sup_{ij}|r_{ij}-r'_{ij}|,\, \sup_i d_B(b_i,b'_i)).
\end{equation}
Similarly, denoting the natural numbers by $\mathbb{N}$, let
\begin{equation}
\calR:=\{(r_{ij}) \in \R^{\mathbb{N} \times \mathbb{N}}: r_{ij}=r_{ji}, r_{ii}=0, r_{ik} \leq r_{ij}+r_{jk}  \}
\end{equation}
be the space of all countably infinite (pseudo) distance matrices equipped with the weak topology; that is, the coarsest topology that makes all projections $\pi_n \colon \calR \to \calR^n$ (onto the northwest $n \times n$ quadrant) continuous, $n >0$.
\begin{definition} \label{def matrix dstrbton}
Let $B$ be a Polish space.
\begin{enumerate}[(i)] 
\item The space of (countably infinite) {\em augmented distance matrices} (ADM) is defined as $\calR_B := \calR \times B^\infty$. 
\item Similarly, for $n> 0$, define the space of $n \times n$ {\em augmented distance matrices} as $\calR_B^n:= \fmm{R}^n \times B^n$. 
\end{enumerate}
\end{definition}
\noindent
In the study of $mm$-fields $(X, B, \pi, \mu)$, if $(x_n)$ is a sequence in $X$, we associate to $(x_n)$ the ADM defined by $r_{ij} = d(x_i, x_j)$ and $b_i = \pi (x_i)$. 

\begin{definition}[ADM Distribution]
Let $\mathcal{X}=(X, B,\pi,\mu)$ be an $mm$-field and $\calFX: X^\infty \to \mathcal{R}_B$ the map $(x_i) \mapsto (d_X(x_i,x_j),\pi(x_i))$. The {\em augmented distance matrix distribution} of $\fmm{X}$ is defined as $\calDX =(\calFX)_*(\mu^\infty)$. Similarly, for $n>0$, define $\calFXn: X^n \to \calR_B^n$ and $\mathcal{D}_\mathcal{X}^n:=(\calFXn)_*(\mu^n)$.
\end{definition}
\begin{theorem}(Field Reconstruction Theorem) \label{thm:reconstruction}
Let $\fmm{X} = (X,d_X,\pi_X, \mu_X)$ and $\fmm{Y} = (Y,d_Y, \pi_Y, \mu_Y)$ be $mm$-fields over $B$ such that $\mu_X$ and $\mu_Y$ are fully supported. Then,
\[
\calX \simeq \calY \text{ if and only if } \calDX=\calDY.
\]
\end{theorem}

\begin{proof}
Clearly, $\calX \simeq \cal Y$ implies that   $\calDX = \calDY$. Suppose that $\calDX = \calDY$. The subset of $U_X \subseteq X^\infty$ of all equidistributed sequences in $X$ is a Borel measurable set of full measure \cite[Lemma~2.4]{kondo2005probability}. Hence, its image $C_X:=\calFX(U_X)$ is an analytical set that has full measure in the completion of $\calDX$ \cite[Theorem~13.2.6]{dudley2018real}. Define $C_Y$ similarly. By construction, $C_X \cap C_Y$ is of full measure in the completion of $\calDX=\calDY$. Let $(x_n) \in U_X$, $(y_n) \in U_Y$ such that $\calFX((x_n))=\calFY((y_n))$. The map $\phi \colon \{x_n \colon n\geq 1\} \to \{y_n \colon n\geq 1\}$ given by $\phi(x_i) = y_i$ is a well-defined isometry satisfying $\piX(x_i)=\piY(y_i)$. Since $(x_n)$ and $(y_n)$ are dense in $X$ and $Y$, respectively, $\phi$ induces an isometry $\Phi \colon \calX \dto \calY$.
\end{proof}

On the space $\calR_B$, we also define the (extended) metric
\begin{equation}
\rho ((r_{ij},b_i),(r'_{ij},b'_i)):=\max (\frac{1}{2}\sup_{ij}|r_{ij}-r'_{ij}|,\, \sup_i d_B(b_i,b'_i)).
\end{equation}
The metric $\rho_n$ over $\calRBn$ is defined similarly. However, since $(\calR_B, \rho)$ is not separable, instead of using $\rho$ to define a topology on $\calR_B$, we only employ it to formulate the Wasserstein distance in $\calR_B$. The next lemma shows that $\rho$ is sufficiently regular for the Wasserstein distance so defined to satisfy some desirable properties. 
\begin{lemma}\label{lemma: lower-semi-continuity of partial}
The extended function $\rho \colon \mathcal{R}_B \times \mathcal{R_B}\to [0,\infty]$ is lower semicontinuous.
\end{lemma}
\begin{proof}
Let $\pi_n: \calRB \to \calRB^n$ denote the projection map. Note that $\rho_n \circ \pi_n \uparrow \rho$ pointwise. Hence, as a pointwise supremum of a sequence of continuous functions, $\rho$ is lower semicontinuous.
\end{proof}

In the discussion below, the $p$-Wasserstein distances $\dwp (\calDX, \calDY)$ and $\dwp (\calDXn, \calDYn)$ are taken with respect to the distance functions $\rho$ and $\rho_n$, respectively.

\begin{theorem}\label{thm: dWp(Dx,Dy)=dGW(infty)}
Let $\calX$ and $\calY$ be bounded $mm$-fields over $B$. Then, for any $1 \leq p \leq \infty$, we have
\begin{equation*}\label{dgw=dwp}
    \lim_{n \to \infty} \dwp(\calDXn,\calDYn) = \dwp(\mathcal{D}_\mathcal{X},\mathcal{D}_\mathcal{Y}) = \dgwi(\mathcal{X},\mathcal{Y}) .
\end{equation*}
\end{theorem}
\begin{proof}
 The projection map $\pi_{n}^{n+1} \colon \calRB^{n+1} \to \calRBn$ is $1$-Lipschitz and has the property that
 \begin{equation}
 (\pi_{n}^{n+1})_\ast (\calDX^{n+1}) = \calDXn \quad \text{and} \quad  (\pi_{n}^{n+1})_\ast (\calDY^{n+1}) = \calDYn \,.
 \end{equation}
 This implies that $\dwp(\calDXn, \calDYn) \leq \dwp(\calDX^{n+1},\calDY^{n+1})$. By a similar argument, we get $\dwp(\calDXn,\calDYn) \leq \dwp(\calDX,\calDY)$. Therefore, we have
\begin{equation}
\begin{split}
    \lim_{n \to \infty} \dwp(\calDXn,\calDYn) = \sup_n \dwp(\calDXn,\calDYn) \leq \dwp(\calDX,\calDY).
\end{split}
\end{equation}
Since $\rho$ is lower semicontinuous by Lemma \ref{lemma: lower-semi-continuity of partial}, using an argument similar to the proof of Proposition \ref{prop:gwrealization}, one can show that $\dwi(\calDX,\calDY) = \lim_{p \to \infty} \dwp(\calDX,\calDY)$. Therefore, without loss of generality, we can assume that $p<\infty$.

Now, we show that $\dwp(\calDX,\calDY) \leq \dgwi(\calX,\calY)$. Let $P$ be the optimal coupling between $\mu_X$ and $\mu_Y$ realizing $\dgwi(\calX,\calY)$. Let $\psiX: (X \times Y)^\infty \to \calRB$ be the map given by $(x_n,y_n)_{n=1}^\infty \mapsto \calFX((x_n)_{n=1}^\infty)$. Define $\psiY$ similarly. Then $Q:=(\psiX,\psiY)_*(P^\infty)$ is a coupling between $\calDX$ and $\calDY$. We have
\begin{equation}
\begin{split}
    \dwp(\calDX,\calDY) &\leq \bigg(\int \rho^p dQ \bigg)^{1/p} = \bigg(\int_{\sP^\infty} \rho^p(\psiX((x_n,y_n)_n),\psiY((x_n,y_n)_n) \dP^\infty ((x_n,y_n)_n) \bigg)^{1/p} \\
    &= \bigg(\int_{(\sP)^\infty} \max\big(\frac{1}{2}\sup_{i,j} \mXY(x_i,y_i,x_j,y_j), \sup_{i} \dXY(x_i,y_i)\big)^p \dP^\infty((x_n,y_n)_n) \bigg)^{1/p} \\
    &\leq \costiP = \dgwi(\calX,\calY).
\end{split}
\end{equation}
It remains to show that $\dgwi(\calX,\calY) \leq \lim_N \dwp(\calDXn,\calDYn)$. Let $0 < \epsilon < 1/2$. By Proposition \ref{prop:gwrealization}, there exists $1\leq q < \infty$ so that $\dgwi(\calX,\calY) \leq \dgw{q}(\calX,\calY) + \epsilon$. Let 
\begin{equation}
\unxqe:=\{(x_i) \in X^n \colon \dwa{q}(\muX, \sum_{i=1}^n \delta_{x_i}/n) \leq \epsilon \} .
\end{equation}
Define $\unyqe$ similarly. By Proposition \ref{prop:almostuniform}, if $n$ large enough, then $\muX^n(\unxqe) \geq 1-\epsilon$ and $\muY^n(\unyqe) \geq 1 - \epsilon$. If we define $\cnxqe:=\calFXn(\unxqe)$ and $\cnyqe:=\calFXn(\unyqe)$, both of these sets are analytical, hence measurable in the completion of $\calDXn$  and $\calDYn$, respectively \cite[Theorem~13.2.6]{dudley2018real}. Moreover, the measures of these sets are $\geq 1-\epsilon$. Let $P$ be the coupling realizing $\dwp(\calDXn,\calDYn)$. Note that we have 
\begin{equation}
\begin{split}
    P(\cnxqe \times \cnyqe) \geq 1 - 2\epsilon.
\end{split}
\end{equation}
By Proposition \ref{prop:dndifference}, we also have 
\begin{equation}
    \rho_n|_{\cnxqe \times \cnyqe} \geq \dgw{q}(\calX,\calY) - 2\epsilon \geq \dgwi(\calX,\calY)-3\epsilon.
\end{equation}
Therefore,
\begin{equation}
    \dwp(\calDXn,\calDYn) = \bigg(\int \rho_n^p \dP \bigg)^{1/p} \geq (\dgwi(\calX,\calY)-3\epsilon)(1-2\epsilon)^{1/p}.
\end{equation}
This implies that 
\begin{equation}
    \lim_{n \to \infty} \dwp(\calDXn,\calDYn) \geq (\dgwi(\calX,\calY)-3\epsilon)(1-2\epsilon)^{1/p}.
\end{equation}
Since $0 < \epsilon < 1/2$ is arbitrary, we get
\begin{equation}
    \lim_{n \to \infty} \dwp(\calDXn,\calDYn) \geq \dgwi(\calX,\calY),
\end{equation}
as claimed.
\end{proof}

\begin{corollary}
Let $P$ be the optimal coupling realizing $\dgwi(\calX,\calY)$. Let $Q$ be the coupling between $\calDX$ and $\calDY$ induced by $P$. Then, for all $p\geq 1$, $Q$ is the optimal coupling realizing $\dwp(\calDX,\calDY)$. Furthermore, if $(x_n,y_n)$ is a uniformly distributed sequence with respect to $P$ in $\sP$, then
\begin{equation}
    \rho (\calFX((x_n)),\calFY((y_n))=\dgwi(\calX,\calY).
\end{equation}
\end{corollary}
\begin{proof}
    Note that $Q$ is used in the proof of Theorem \ref{thm: dWp(Dx,Dy)=dGW(infty)} to show $\dwp(\calDX,\calDY)\leq \dgwi(\calX,\calY)$. But since we have equality, $Q$ is the optimal coupling. The equality follows from the proof of Theorem \ref{thm:gwuniform}.
\end{proof}

\begin{remark}\label{rem:discretezation}
By Theorem $\ref{thm: dWp(Dx,Dy)=dGW(infty)}$, $\dwp(\calDXn,\calDYn)$ can be used as an approximation to $\dgwi(\calX,\calY)$. To discretize this approximation, one can take i.i.d. samples from $(X^n,\mu^n)$ and $(Y^n,\nu^n)$ and form empirical measures $\calE_{n,X}$ and $\calE_{n,Y}$. Then, $\dwp(((\calFXn)_*(\calE_{n,X}), (\calFYn)_*(\calE_{n,Y}) )$ can be taken as an approximation to $\dwp(\calDXn,\calDYn)$.
\end{remark}

\begin{remark}\label{rem:p-independence}
By Theorem \ref{thm: dWp(Dx,Dy)=dGW(infty)}, $\dwp(\calDX,\calDY)$ is independent of $p \geq 1$. This can be explained as follows. Since we are using the sup-distance $\rho$ on $\calRB$ and almost every sequence in a metric measure space is uniformly distributed, if $\dwp(\calDX,\calDY) \leq r$, then there are uniformly distributed sequences in $\calX, \calY$ whose augmented distance matrices are $r$-close to each other, which forces $\dwa{q}(\calDX,\calDY) \leq r$ for any $q \geq 1$.
\end{remark}

The following theorem gives a geometric representation of the isomoprhism classes of $mm$-fields via the Urysohn field.

\begin{theorem}\label{prop:mmrespresentation}
Let $\fmm{I}_B$ denote the moduli space of isometry classes of compact $mm$-fields over $B$ endowed with the distance $\dgwi$. Let $\fmm{G}_B$ be the group of automorphisms of the Urysohn field $\fmm{U}_B$ and $\lawsB$ be the set of compactly supported laws on $\urysohnB$, endowed with the distance $\dwi$. The group $\fmm{G}_B$ acts on $\lawsB$ by $g\cdot \mu:=g_*(\mu)$. Then, $\fmm{I}_B$ is isomorphic to the orbit space of this action; that is, $\fmm{I}_B\simeq \lawsB / \fmm{G}_B$, where the orbit space is equipped with the quotient metric, as in \eqref{eq:autb}, which can be expressed as
\begin{equation*}
        d([\mu],[\nu])=\inf_{g \in \fmm{G}_B} \dwi(\mu,g_*(\nu)).
\end{equation*} 
\end{theorem}
\begin{proof}
Given a compact $mm$-field $\calX$ over $B$ and an isometric embedding $\iota: \calX \to \urysohnB$, we have $(\iota)_*(\muX) \in \lawsB$. This induces a well-defined map from $\Psi \colon \fmm{I}_B \to \lawsB / \fmm{G}_B$ because of the compact homogeneity of $\urysohnB$. 
Therefore, by Proposition \ref{prop:gwembedding}, $\dgwi(\calX,\calY) \leq d(\Psi(\calX),\Psi(\calY))$. To show the opposite inequality, let $P$ be the optimal coupling realizing $\dgwi(\calX,\calY)$. Let $\calZ$ be the $B$-field constructed using $P$ as in Proposition \ref{prop:couplingembedding}. Consider $\muX \sqcup \muY$ as a measure on $Z$ and let $\iota \colon Z \to \urysohnB$ be an isometric embedding. Then,
\begin{equation}
            d(\Psi(\calX),\Psi(\calY)) \leq \dwi(\iota_*(\muX),\iota_*(\muY)) \leq \dwi^Z(\muX,\muY) \leq \dgwi(\calX,\calY),
\end{equation}
establishing the desired inequality.
\end{proof}


\section{Summary and Discussion} \label{sec:summary}

This paper studied functional data, termed fields, defined on geometric domains. More precisely, the objects of study were 1-Lipschitz functions between Polish metric spaces with the domain possibly equipped with a Borel probability measure. We addressed foundational questions and developed new approaches to the analysis of datasets consisting of fields not necessarily defined on the same domains. We constructed Urysohn fields for a fixed target Polish space $B$; that is, universal and homogeneous objects for fields over $B$. We also constructed ultimate Urysohn fields that have analogous universality and homogeneity properties, but allow the target Polish space to also vary. We defined and investigated the basic properties of a Gromov-Hausdorff distance between compact fields and how it relates to the Hausdorff distance in the Urysohn field via functional isometric embeddings. Similarly, we studied a notion of Gromov-Wasserstein distance between fields with metric-measure domains. 

We also introduced a representation of metric-measure fields as probability distributions on the space of (countably) infinite augmented distance matrices and proved a Reconstruction Theorem that extends to $mm$-fields a corresponding result for $mm$-spaces due to Gromov. This provides a pathway to discrete representations of $mm$-fields via distributions of finite-dimensional augmented distance matrices for which we proved a convergence theorem.

Questions that also are of interest but fall beyond the scope of this paper include: (i) the study of rate of convergence of the probabilistic model based on finite-dimensional augmented distance matrices; (ii) investigation of alternative cost functions in the formulation of the Gromov-Wasserstein distance between $mm$-fields; (iii) the development of computational models derived from augmented distance matrices.

\section*{Acknowledgements}

This work was partially supported by NSF grant DMS-1722995.

\appendix
\section{Appendix}\label{sec:appendix}

\begin{proof}[Proof of Lemma \ref{compact injectivity}]
It suffices to construct a sequence $(u_n)$ in $U$ satisfying
\begin{enumerate}[(i)]
\item $\pi_U(u_n) = f(a^\ast)$, for $n \geq 1$;
\item $|d_\ast(a^\ast, a)-d_U(u_n,a)|\leq 2^{-n}$, $\forall a\in A$, where $d_\ast$ is the metric on $A^\ast$;
\item $d_U(u_n,u_{n+1})\leq 2^{-n}$, for $n \geq 1$.
\end{enumerate}
Indeed, letting $u=\lim_{n\to \infty} u_n \in U$, define $\imath \colon A^\ast \to B$ as the identity on $A$ and $\imath (a^\ast) =u$. The map $\imath$ gives the desired one-point extension. Now we proceed to the construction of the sequence $(u_n)$.

Let $A_n$ be an ascending sequence of finite subsets of $A$, where $A_n$ is a $2^{-(n+1)}$-net in $A$, for $n \geq 1$. (This means that the balls of radius $2^{-(n+1)}$ centered at the points in $A_n$ cover $A$.) Let $D_1=A_1$ and denote by $D_1^\ast = D_1\sqcup\{a^\ast\}$, the one-point metric extension of $D_1$ induced by $(A^\ast, d_\ast)$. Since $\mm{U}_B$ is Urysohn, applying the one-point extension property to the field $f|_{D_1^\ast}$, we obtain an isometric embedding $\imath_1 \colon D_1^\ast \to U$ such that $\pi_U \circ \imath_1 = f|_{D_1^\ast}$. Defining $u_1 = \imath_1 (a^\ast) \in U$, we have that $\pi_U (u_1) = f(a^\ast)$ and $d_\ast (a^\ast, a) = d_U (u_1, a)$, for any $a \in A$, so $u_1$ satisfies (i) and (ii) above. Condition (iii) is empty at this stage of the construction. Inductively, suppose that we have constructed  $u_j$, $i \leq j \leq n$, with the desired properties and let
\begin{equation}
D_{n+1} = A_{n+1} \cup \{u_n\} \quad \text{and} \quad D_{n+1}^\ast=  D_{n+1} \sqcup\{a^\ast\}.
\end{equation}
Using the notation $A_{n+1}^\ast = A_{n+1} \cup \{a^\ast\}$, define a metric $d'_\ast \colon D_{n+1}^\ast \times D_{n+1}^\ast \to \real$, as follows: $d'_\ast$ coincides with $d_\ast$ on $A_{n+1}^\ast \times A_{n+1}^\ast $, $d'_\ast (a,u_n) = d_U (a, u_n)$, for every $a \in A_{n+1}$, and
\begin{equation}
d'_\ast(a^\ast, u_n):=\sup_{b\in A_{n+1}}|d_\ast(a^\ast,b)-d_U(u_n,b)|.
\label{eq:metric}
\end{equation}
Define a field $f' \colon D_{n+1}^\ast \to B$ by $f'|_{D_{n+1}} = \pi_U|_{D^{n+1}}$ and $f'(a^\ast) = f(a^\ast)$. Applying the one-point extension property to $f'$ we obtain an isometric embedding $\imath_{n+1} \colon D_{n+1}^\ast \to U$ satisfying $f' = \pi_U\circ \imath_{n+1}$.

Let $u_{n+1} = \imath_{n+1} (a^\ast) \in U$. By construction, $\pi_U(u_{n+1}) = f'(a_\ast) = f (a_\ast)$, so requirement (i) is satisfied. Moreover, $d_U(u_{n+1},b) = d'_\ast(a^\ast, b)= d_\ast(a^\ast, b)$, for any $b \in A_{n+1}$. Since $A_{n+1}$ is a $2^{-(n+2)}$-net in $A$, for each $a \in A$, we can pick $b \in A_{n+1}$ such that $d(a,b) \leq 2^{-(n+2)}$. Then, we have
\begin{equation}
\begin{split}
|d_\ast (a^\ast, a) - d_U(u_{n+1},a)| &\leq |d_\ast (a^\ast, a) - d_U(u_{n+1},b)| + | d_U(u_{n+1},b) - d_U(u_{n+1},a)| \\
&= |d_\ast (a^\ast, a) - d_\ast (a^\ast,b)| + | d_U (u_{n+1},b) - d_U(u_{n+1},a)| \\
&\leq d_\ast(a,b) + d_U (a,b) = 2 d_U (a,b) \leq 2^{-(n+1)} \,.
\end{split}
\end{equation}
This verifies property (ii). By the inductive hypothesis, we also have $|d_\ast (a^\ast, a) - d_U(u_n,a)| \leq 2^{-n}$, for any $a \in A$. Thus, by \eqref{eq:metric},
\begin{equation}
d(u_{n+1}, u_n) = d'_\ast (a^\ast, u_n) = \sup_{b\in A_{n+1}}|d_\ast(a^\ast,b)-d_U(u_n,b)| \leq 2^{-n}.
\end{equation}
This concludes the proof.
\end{proof}

\begin{proposition}\label{prop:couplings}    
    Every sequence in $\cXY$ has a weakly convergent subsequence.
\end{proposition}
\begin{proof}
By \cite[Theorem~9.3.7]{dudley2018real}, $\cXY$ is closed under weak convergence. By Prokhorov's Theorem \cite[Theorem~11.5.4]{dudley2018real}, it is enough to show that $C(\muX,\muY)$ is uniformly tight. Since $X$ and $Y$ are Polish, $\muX$ and $\muY$ are tight measures (\cite[Theorem~7.1.4]{dudley2018real}). Let $\epsilon>0$. There are compact subspaces $K_X \subseteq X$ and $K_Y \subseteq Y$ so that $\muX(K_X)>1-\epsilon/2$ and $\muY(K_Y)>1-\epsilon/2$. Then, for any $P$ in $\cXY$, we have 
    \begin{equation}
    \begin{split}
        P(K_X \times K_Y) &= P((K_X \times Y) \cap (X \times K_Y)) \\
        &\geq P(K_X \times Y) + P(X \times K_Y) -1 \\
        &= \muX(K_X) + \muY(K_Y)-1 \geq 1-\epsilon.
    \end{split}
    \end{equation}
    This concludes the proof.
\end{proof}

\begin{proposition}\label{prop:almostuniform}
Let $1\leq p < \infty$ and $(X,\dX,\muX)$ be a metric measure space such that $\mu_X$ has finite moments of order $p$, where $1 \leq p < \infty$. Given an integer  $n>0$ and $\epsilon>0$, let 
\begin{equation*}
\unxpe:=\{(x_i) \in X^n \colon \dwp(\muX, \sum_{i=1}^n \delta_{x_i}/n) \leq \epsilon \} .
\end{equation*}
Then, for $n$ sufficiently large, we have
 \begin{equation*}
 \mu^N(\unxpe) \geq 1 - \epsilon.
 \end{equation*}
\end{proposition}
\begin{proof}
 Let $P_p(X)$ denote the set of Borel probability measures on $X$ with finite moments of order $p$. $P_p(X)$  is metrizable by $\dwp$, and the corresponding notion of convergence is weak convergence \cite[Theorem~6.9]{villani2008optimal}. Furthermore, $P_p(X)$ is complete and separable \cite[Theorem~6.18]{villani2008optimal}. Let $\pi_n \colon X^\infty \to X^n$ denote the projection onto the first $n$ coordinates and $\psi_n \colon X^\infty \to P_p(X)$ be the map given by
 \begin{equation}
\psi_n((x_i)):= \sum_{i=1}^n \delta_{x_i}/n.
 \end{equation}
 This is a continous map. By Varadarajan Theorem \cite[Theorem~11.4.1]{dudley2018real}, $(\psi_n)$ converges to $\muX$ almost surely. By \cite[Theorem~9.2.1]{dudley2018real}, $(\psi_n)$ converges to $\mu$ in probability. Hence, for $n$ large enough, 
\begin{equation}
1 - \epsilon \leq \mu^\infty(\dwp(\muX,\psi_n) \leq \epsilon) = \mu^\infty(\pi_n^{-1}(\unxpe))=\mu^n(\unxpe),
\end{equation}
as desired.
\end{proof}

\begin{proposition}\label{prop:dndifference}
Let $\calX$ and $\calY$ be $mm$-fields over $B$. For $\epsilon >0$, let $\unxpe$ and $\unype$ be defined as in Proposition \ref{prop:almostuniform}. If $(x_i) \in \unxpe$, and $(y_i) \in \unype$, then
 \begin{equation*}
 \rho_n (\calFXn((x_n)), \calFYn((y_n))) \geq \dgwp(\calX,\calY) - 2\epsilon.
 \end{equation*}
\end{proposition}
\begin{proof}
Let $\calE_X$ be the $mm$-field with underlying set $\{1,\dots,n \}$ equipped with the (pseudo)-metric given by $d_E(i,j)=\dX(x_i,x_j)$, normalized counting measure, and $1$-Lipschitz function $i \mapsto \piX(x_i)$. Similarly, define $\calE_Y$ using $(y_n)$. Note that, by Proposition \ref{prop:gwembedding}, $\dgwp(\calE_X,\calX) \leq \epsilon$ and $\dgwp(\calE_Y,\calY) \leq \epsilon$. If $P$ is the diagonal coupling between the measures of $\calE_X$ and $\calE_Y$, then we have
\begin{equation}
\begin{split}
        \rho_n (\calFXn((x_i)), \calFYn((y_i))) &= \max \left\{  \frac{1}{2} \sup_{\sPP} m_{E_X,E_Y}, \sup_{\sP} d_{E_X,E_Y} \right\} \\
        &\geq \dgwp(\calE_X,\calE_Y) \geq \dgwp(\calX,\calY) - 2\epsilon\,,
\end{split}
\end{equation}  
as claimed.
\end{proof}

\begin{proposition}\label{prop:couplingembedding}
Let $\calX$, $\calY$ be $mm$-fields over $B$ and $P$ be a coupling between $\muX$ and $\muY$, and 
\begin{equation*}
r \colon =\costiP < \infty.
\end{equation*}
 Let $Z= X \sqcup Y$ and define $\pi_Z\colon Z \to B$ by $\piZ|_X:=\piX$ and $\piZ|_Y=\piY$. If $\dZ \colon Z \times Z \to [0,\infty)$ is given by $\dZ|_{X \times X}\colon=\dX$, $\dZ|_{Y \times Y} \colon= \dY$, and 
 \begin{equation*}
 \dZ(x,y)=\dZ(y,x):=r+\inf_{(x',y') \in \sP}\dX(x,x')+\dY(y,y') \,
 \end{equation*}
 for $x \in X$ and $y \in Y$, then $\dZ$ is a metric on $Z$ and $\piZ$ is $1$-Lipschtiz with respect to $\dZ$. Furthermore, $\dwi(\muX,\muY) \leq r$, where the Wasserstein distance is evaluated in $Z$.
\end{proposition}
\begin{proof}
The function $\dZ$ is clearly symmetric. Let us show that it satisfies the triangle inequality. Let $x_1, x_2 \in X$ and $y \in Y$.
    \begin{equation}
    \begin{split}
        \dZ(x_1,y)+\dZ(y,x_2) &= 2r + \inf_{(x',y',x'',y'') \in \sPP} \dX(x_1,x')+\dY(y',y) + \dY(y,y'') + \dX(x_2,x'') \\
        & \geq 2r + \inf_{(x',y',x'',y'') \in \sPP} \dX(x_1,x_2)-\dX(x',x'')+\dY(y',y'') \\
        & \geq \dX(x_1,x_2) + \inf_{(x',y',x'',y'') \in \sPP} \mXY(x',y',x'',y'')-(\dX(x',x'')-\dY(y',y''))\\
        & \geq \dX(x_1,x_2).
    \end{split}
    \end{equation}
The remaining cases of the triangle inequality are either similar or straightforward. To show that $\piZ$ is $1$-Lipschitz, we only need to check the case $x \in X$ and $y \in Y$, for which we have
\begin{equation}
\begin{split}
            \dB(\piZ(x),\piZ(y))&=\dB(\piX(x),\piY(y)) \\
            &\leq \inf_{(x',y') \in \sP} \dB(\piX(x),\piX(x')) + \dB(\piX(x'),\piY(y')) + \dB(\piY(y'),\piY(y)) \\
            &\leq \inf_{(x',y') \in \sP} \dX(x,x') + r + \dY(y,')= \dZ(x,y).
\end{split}
\end{equation}
Viewing $\muX \sqcup \muY$ as a measure on $Z$, we have
\begin{equation}
            \dwi(X,Y) \leq \inf_{(x',y') \in \sP} \dZ(x',y')=r.
\end{equation}
This concludes the proof.
\end{proof}


\bibliographystyle{siam}
\bibliography{realbib}

\begin{thebibliography}{10}

\bibitem{banach1987theory}
{\sc S.~Banach}, {\em Theory of Linear Operations}, Elsevier, 1987.

\bibitem{billingsley2013convergence}
{\sc P.~Billingsley}, {\em Convergence of Probability Measures}, John Wiley \&
  Sons, 2013.

\bibitem{burago2001course}
{\sc D.~Burago, Y.~Burago, and S.~Ivanov}, {\em A Course in Metric Geometry},
  vol.~33, American Mathematical Soc., 2001.

\bibitem{coifman}
{\sc R.~R. Coifman and S.~Lafon}, {\em Diffusion maps}, Applied and
  Computational Harmonic Analysis, 21 (2006), pp.~5--30.

\bibitem{doreian2013evolution}
{\sc P.~Doreian and F.~Stokman}, {\em Evolution of Social Networks}, Routledge,
  2013.

\bibitem{doucha2013universal}
{\sc M.~Doucha}, {\em Universal and ultrahomogeneous {P}olish metric
  structures}, arXiv:1305.0501,  (2013).

\bibitem{dudley2018real}
{\sc R.~M. Dudley}, {\em Real Analysis and Probability}, CRC Press, 2018.

\bibitem{frechet1910dimensions}
{\sc M.~Fr{\'e}chet}, {\em The dimensions of an abstract set}, Mathematische
  Annalen, 68 (1910), pp.~145--168.

\bibitem{givens1984class}
{\sc C.~R. Givens and R.~M. Shortt}, {\em A class of {W}asserstein metrics for
  probability distributions.}, Michigan Mathematical Journal, 31 (1984),
  pp.~231--240.

\bibitem{gromov2007metric}
{\sc M.~Gromov}, {\em Metric {S}tructures for Riemannian and {N}on-Riemannian
  {S}paces}, Springer Science \& Business Media, 2007.

\bibitem{hang2019topological}
{\sc H.~Hang, F.~M{\'e}moli, and W.~Mio}, {\em A topological study of
  functional data and {F}r{\'e}chet functions of metric measure spaces},
  Journal of Applied and Computational Topology, 3 (2019), pp.~359--380.

\bibitem{heinonen2003geometric}
{\sc J.~Heinonen}, {\em Geometric {E}mbeddings of {M}etric {S}paces}, 2003.

\bibitem{katetov1986universal}
{\sc M.~Katetov}, {\em On universal metric spaces}, in General Topology and Its
  Relations to Modern Analysis and Algebra IV (Prague, 1986), Research and
  Exposition in Mathematics, vol.~16, 1986, pp.~323--330.

\bibitem{kechris2005fraisse}
{\sc A.~S. Kechris, V.~G. Pestov, and S.~Todorcevic}, {\em Fra{\"\i}ss{\'e}
  limits, {R}amsey theory, and topological dynamics of automorphism groups},
  Geometric \& Functional Analysis GAFA, 15 (2005), pp.~106--189.

\bibitem{kondo2005probability}
{\sc T.~Kondo}, {\em Probability distribution of metric measure spaces},
  Differential Geometry and its Applications, 22 (2005), pp.~121--130.

\bibitem{martinez2019probing}
{\sc D.~H.~D. Martinez, C.~H. Lee, P.~T. Kim, and W.~Mio}, {\em Probing the
  geometry of data with diffusion {F}r{\'e}chet functions}, Applied and
  Computational Harmonic Analysis, 47 (2019), pp.~935--947.

\bibitem{melleray2007geometry}
{\sc J.~Melleray}, {\em On the geometry of {U}rysohn's universal metric space},
  Topology and its Applications, 154 (2007), pp.~384--403.

\bibitem{melleray2008some}
\leavevmode\vrule height 2pt depth -1.6pt width 23pt, {\em Some geometric and
  dynamical properties of the {U}rysohn space}, Topology and its Applications,
  155 (2008), pp.~1531--1560.

\bibitem{memoli2011gromov}
{\sc F.~M{\'e}moli}, {\em Gromov--{W}asserstein distances and the metric
  approach to object matching}, Foundations of Computational Mathematics, 11
  (2011), pp.~417--487.

\bibitem{mrowka1953solution}
{\sc S.~Mr{\'o}wka}, {\em Solution of a {U}rysohn problem concerning universal
  metric spaces}, Bull. Acad. Polon. Sci, 1 (1953), pp.~233--234.

\bibitem{sekara2016fundamental}
{\sc V.~Sekara, A.~Stopczynski, and S.~Lehmann}, {\em Fundamental structures of
  dynamic social networks}, Proceedings of the National Academy of Sciences of
  the United States of America, 113 (2016), pp.~9977--9982.

\bibitem{sturm2006geometry}
{\sc K.-T. Sturm et~al.}, {\em On the geometry of metric measure spaces}, Acta
  Mathematica, 196 (2006), pp.~65--131.

\bibitem{tantipathananandh2007framework}
{\sc C.~Tantipathananandh, T.~Berger-Wolf, and D.~Kempe}, {\em A framework for
  community identification in dynamic social networks}, in Proceedings of the
  13th ACM SIGKDD International Conference on Knowledge Discovery and Data
  Mining, 2007, pp.~717--726.

\bibitem{urysohn1927espace}
{\sc P.~Urysohn}, {\em Sur un espace m{\'e}trique universel}, Bull. Sci. Math,
  51 (1927), pp.~43--64.

\bibitem{vayer2020fused}
{\sc T.~Vayer, L.~Chapel, R.~Flamary, R.~Tavenard, and N.~Courty}, {\em Fused
  {G}romov-{W}asserstein distance for structured objects}, Algorithms, 13
  (2020), p.~212.

\bibitem{vershik2004random}
{\sc A.~M. Vershik}, {\em Random and universal metric spaces}, in Dynamics and
  Randomness II, Springer, 2004, pp.~199--228.

\bibitem{villani2008optimal}
{\sc C.~Villani}, {\em Optimal Transport: Old and New}, vol.~338, Springer
  Science \& Business Media, 2008.

\bibitem{yaacov2015fraisse}
{\sc I.~B. Yaacov}, {\em Fra{\"\i}ss{\'e} limits of metric structures}, The
  Journal of Symbolic Logic, 80 (2015), pp.~100--115.

\end{thebibliography}

\end{document}